\documentclass[leqno]{amsart}
  
\usepackage{amssymb}
\usepackage{amsthm}
\usepackage{amscd}
\usepackage{mathrsfs}

\usepackage[all]{xy}
\usepackage{hyperref} 

\begin{document}

\bibliographystyle{alpha}
\newtheorem{proposition}[subsubsection]{Proposition}
\newtheorem{lemma}[subsubsection]{Lemma}
\newtheorem{corollary}[subsubsection]{Corollary}
\newtheorem{thm}[subsubsection]{Theorem}
\newtheorem{introthm}{Theorem}
\newtheorem*{thm*}{Theorem}
\newtheorem{conjecture}[subsubsection]{Conjecture}
\newtheorem{question}[subsubsection]{Question}
\newtheorem{fails}[subsubsection]{Fails}

\theoremstyle{definition}
\newtheorem{definition}[subsubsection]{Definition}
\newtheorem{notation}[subsubsection]{Notation}
\newtheorem{condition}[subsubsection]{Condition}
\newtheorem{example}[subsubsection]{Example}
\newtheorem{claim}[subsubsection]{Claim}

\theoremstyle{remark}
\newtheorem{remark}[subsubsection]{Remark}

\numberwithin{equation}{subsection}

\newcommand{\eq}[2]{\begin{equation}\label{#1}#2 \end{equation}}
\newcommand{\ml}[2]{\begin{multline}\label{#1}#2 \end{multline}}
\newcommand{\mlnl}[1]{\begin{multline*}#1 \end{multline*}}
\newcommand{\ga}[2]{\begin{gather}\label{#1}#2 \end{gather}}
\newcommand{\mat}[1]{\left(\begin{smallmatrix}#1\end{smallmatrix}\right)}

\newcommand{\arir}{\ar@{^{(}->}}
\newcommand{\aril}{\ar@{_{(}->}}
\newcommand{\are}{\ar@{>>}}

\newcommand{\xr}[1] {\xrightarrow{#1}}
\newcommand{\xl}[1] {\xleftarrow{#1}}
\newcommand{\lra}{\longrightarrow}
\newcommand{\inj}{\hookrightarrow}

\newcommand{\mf}[1]{\mathfrak{#1}}
\newcommand{\mc}[1]{\mathcal{#1}}

\newcommand{\CH}{{\rm CH}}
\newcommand{\Gr}{{\rm Gr}}
\newcommand{\codim}{{\rm codim}}
\newcommand{\cd}{{\rm cd}}
\newcommand{\Spec} {{\rm Spec}}
\newcommand{\supp} {{\rm supp}}
\newcommand{\Hom} {{\rm Hom}}
\newcommand{\End} {{\rm End}}
\newcommand{\id}{{\rm id}}
\newcommand{\Aut}{{\rm Aut}}
\newcommand{\sHom}{{\rm \mathcal{H}om}}
\newcommand{\Tr}{{\rm Tr}}
\newcommand{\iHom}{\underline{{\rm Hom}}}
\newcommand{\coker}{{\rm coker}}

\renewcommand{\P} {\mathbb{P}}
\newcommand{\Z} {\mathbb{Z}}
\newcommand{\Q} {\mathbb{Q}}
\newcommand{\C} {\mathbb{C}}
\newcommand{\F} {\mathbb{F}}
\newcommand{\W}{\mathbb{W}}

\newcommand{\OO}{\mathcal{O}}

\title[De Rham-Witt cohomology]{Big de Rham-Witt cohomology: basic results}

\author{Andre Chatzistamatiou}
\address{Fachbereich Mathematik \\ Universit\"at Duisburg-Essen \\ 45117 Essen, Germany}
\email{a.chatzistamatiou@uni-due.de}

\thanks{This work has been supported by the SFB/TR 45 ``Periods, moduli spaces and arithmetic of algebraic varieties''}

\begin{abstract}
Let $X$ be a smooth projective $R$-scheme, where $R$ is a smooth $\Z$-algebra. As constructed by Hesselholt, we have the absolute big de Rham-Witt complex $\W\Omega^*_X$ of $X$ at our disposal. There is also a relative version $\W\Omega^*_{X/R}$ with $\W(R)$-linear differential.
In this paper we study the hypercohomology of the relative (big) de Rham-Witt complex after truncation with finite truncation sets $S$. 
We show that it is a projective 
$\W_S(R)$-module, provided that the de Rham cohomology is a flat $R$-module. In addition, we establish a Poincar\'e duality theorem.
\end{abstract}

\maketitle


\section*{Introduction}
Let $X$ be a scheme over a perfect field $k$ of characteristic $p>0$. 
The de Rham-Witt complex $W\Omega^*_{X/k}$ was defined by Illusie \cite{Illusie}
relying on ideas of Bloch, Deligne and Lubkin. It is a projective system of 
complexes of $W(k)$-modules on $X$, which is indexed by the positive integers. 
If $X$ is smooth then the hypercohomology of $W_n\Omega^*_{X/k}$ admits a natural 
comparison isomorphism to the crystalline cohomology of $X$ with respect to $W_n(k)$. 

Langer and Zink have extended Illusie's definition of the de Rham-Witt complex 
to a relative situation, where $X$ is a scheme over $\Spec(R)$ and $R$ is a $\Z_{(p)}$-algebra \cite{LZ}. If $p$ is nilpotent in $R$ and $X$ is smooth, then they construct a functorial 
comparison isomorphism 
$$
H^*(X,W_n\Omega^*_{X/R})\cong H^*_{crys}(X/W_n(R)).
$$

The big de Rham-Witt complex $\W\Omega^*_{A}$ was introduced, for any commutative ring $A$, by Hesselholt and 
Madsen \cite{HM}. The original construction relied on the adjoint functor theorem and 
has been replaced by a direct and explicit method due to Hesselholt \cite{H}. 

Again, it is a projective system of graded sheaves $[S\mapsto \W_S\Omega^*_A]$, but the index set consists 
of finite truncation sets; that is, finite subsets $S$ of $\mathbb{N}_{>0}$ having the 
property that whenever $n\in S$, all (positive) divisors of $n$ are also contained in $S$.
For the  ring of integers,  $\W\Omega^*_{\Z}$  has been computed by Hesselholt \cite{H}. It vanishes in degree $\geq 2$, but $\W\Omega^1_{\Z}$ is non-zero. 

Let $X$ be an $R$-scheme. In this paper we will consider the relative version 
$$
S\mapsto \W_S\Omega^*_{X/R}
$$ 
of the (big) de Rham-Witt complex, which is constructed from   
$\W\Omega^*_{X}$ by killing the ideal generated by $\W\Omega^1_{R}$. 
The relation with the de Rham-Witt complex of Langer-Zink is given in 
Proposition~\ref{proposition-comparison-Langer-Zink}: if $R$ is a $\Z_{(p)}$-algebra then
$$ 
\W_{\{1,p,\dots, p^{n-1}\}}\Omega^*_{A/R} = W_n\Omega^*_{A/R}.
$$  
In the following we will use the notation $W_n=\W_{\{1,p,\dots, p^{n-1}\}}$, assuming that a prime $p$ has been fixed.

It is natural to consider $\W_S\Omega^*_{X/R}$ as a sheaf of complexes on the scheme $\W_S(X)$, which is obtained by gluing $\Spec(\W_S(A_i))$
for an affine covering $X=\bigcup_i \Spec(A_i)$. Then the components $\W_S\Omega^q_{X/R}$ form quasi-coherent sheaves, 
and are coherent under suitable finiteness conditions.   

Our purpose is to show that the de Rham-Witt cohomology 
$$
H^i_{dRW}(X/\W_S(R)) \overset{{\rm def}}{=} H^i(\W_S(X),\W_S\Omega^*_{X/R}) 
$$
is as well-behaved as the usual de Rham cohomology. The main theorem of the paper
is the following.

\begin{introthm}[cf.~Theorem~\ref{thm-projective-blue}]\label{introthm-projective-blue}
   Let $R$ be a smooth $\Z$-algebra. Let $X$ be a smooth and proper $R$-scheme. Suppose that the de Rham cohomology $H^*_{dR}(X/R)$ of $X$
is a flat $R$-module. Then $H^*_{dRW}(X/\W_S(R))$ is a finitely generated 
projective $\W_S(R)$-module for all finite truncation sets $S$. Moreover, for an inclusion of finite truncation sets $T\subset S$, the induced map
\begin{equation}
H^*_{dRW}(X/\W_S(R))\otimes_{\W_S(R)}\W_T(R)\xr{\cong} H^*_{dRW}(X/\W_T(R))  
\end{equation}
is an isomorphism.
\end{introthm}


In order to prove Theorem \ref{introthm-projective-blue}, we will construct
for all maximal ideals $\mf{m}$ of $R$ 
and $n,j>0$, a natural quasi-isomorphism:
$$
R\Gamma(W_n\Omega^*_{X/R})\otimes^{\mathbb{L}}_{W_n(R)}W_n(R/\mf{m}^j)\xr{{\rm q-iso}} R\Gamma(W_n\Omega^*_{X\otimes R/\mf{m}^j/(R/\mf{m}^j)}),
$$
where $p={\rm char}(R/\mf{m})$. The right hand side is $R\Gamma$ of 
the de Rham-Witt complex defined by Langer and Zink. Thus it 
computes the crystalline cohomology, which in our case is a free 
$W_n(R/\mf{m}^j)$-module. Taking the limit $\varprojlim_j$, this will yield the 
flatness of 
$$
H^*_{dRW}(X/W_n(R))\otimes_{W_n(R)} W_n(\varprojlim_j R/\mf{m}^j)
$$
as $W_n(\varprojlim_j R/\mf{m}^j)$-module for all maximal ideals $\mf{m}$, which is sufficient 
in order to conclude the flatness of the de Rham-Witt cohomology.  

Concerning Poincar\'e duality we will show the following theorem.
\begin{introthm}[cf.~Corollary~\ref{corollary-Poincare-duality-made-simple}]
  \label{introthm-2} 
  Let $R$ be a smooth $\Z$-algebra. Let $X\xr{} \Spec(R)$ be a smooth projective morphism such that $H^*_{dR}(X/R)$ is a flat $R$-module. 
Suppose that $X$ is connected of relative dimension $d$. 
 If the canonical map 
\begin{equation*}
H^{i}_{dR}(X/R)\xr{} \Hom_R(H^{2d-i}_{dR}(X/R),R)  
\end{equation*}
is an isomorphism, then the same holds for the de Rham-Witt cohomology:
\begin{equation*} 
H^{i}_{dRW}(X/\W_S(R))\xr{\cong} \Hom_{\W_S(R)}(H^{2d-i}_{dRW}(X/\W_S(R)),\W_S(R)),
\end{equation*}
for all finite truncation sets $S$.
\end{introthm}

In fact, de Rham-Witt cohomology is equipped with a richer structure than 
the $\W(R)$-module structure, coming from the Frobenius operators 
$$
\phi_n:H^*_{dRW}(X/\W_S(R)) \xr{}  H^*_{dRW}(X/\W_{S/n}(R)), 
$$
for all positive integers $n$, and where $S/n:=\{s\in S\mid ns\in S\}$. These
are Frobenius linear maps 
 satisfying $\phi_n\circ \phi_m=\phi_{nm}$. 

The relationship with the Frobenius
action on the crystalline cohomology of the fibers is as follows. 
Let $\mf{m}$ be a maximal ideal of $R$, set $k=R/\mf{m}$ and $p={\rm char}(k)$. If $H^*_{dR}(X/R)$ is torsion-free  then
there is a  natural isomorphism
$$
H^i_{dRW}(X/W_n(R))\otimes_{W_n(R)} W_n(k) \cong H^i_{crys}(X\otimes_R k/W_n(k)),
$$
and $\phi_p\otimes F_p$ corresponds via this isomorphism  to the composition of $H^i_{crys}({\rm Frob})$ with the projection.

As will be made precise in Section \ref{section-values}, the projective system 
$$
H^i_{dRW}(X/\W(R))\overset{{\rm def}}{=} [S\mapsto H^i_{dRW}(X/\W_S(R))],
$$ 
together with the Frobenius morphisms $\{\phi_n\}_{n\in \mathbb{N}_{>0}}$, defines an object 
in a rigid $\otimes$-category $\mathcal{C}_R$.
Maybe the most important property of $\mathcal{C}_R$ is 
the existence of a conservative, faithful $\otimes$-functor 
$$
T:\mathcal{C}_R \xr{} \text{($R$-modules)}, \qquad T(H^i_{dRW}(X/\W(R)))=H^i_{dR}(X/R).
$$
Moreover, $\mathcal{C}_R$ has Tate objects $\mathbf{1}(m)$, $m\in \Z$, and the 
first step towards Poincar\'e duality will be to prove the existence of a natural
morphism in $\mathcal{C}_R$:
$$
H^{2d}_{dRW}(X/\W(R))\xr{}\mathbf{1}(-d) \qquad (d=\text{relative dimension of $X/R$}).
$$  
Then it will  follow easily that 
$$
H^{i}_{dRW}(X/\W(R))\xr{\cong} \iHom(H^{2d-i}_{dRW}(X/\W(R)),\mathbf{1}(-d)), 
$$ 
provided that the assumptions of Theorem \ref{introthm-2} are satisfied. Taking
the underlying $\W(R)$-modules one obtains Theorem \ref{introthm-2}.

\subsection*{Acknowledgements}
After this manuscript had appeared on arXiv, we received a letter from professor James Borger who informed us that he had already obtained 
Theorem \ref{introthm-projective-blue}, for $R=\Z[N^{-1}]$, in a joint work with Mark Kisin by using similar methods.

I thank Andreas Langer and Kay R\"ulling for several useful comments on the first version of the paper.
\tableofcontents

\section{Relative de Rham-Witt complexes}

\subsection{Witt vectors}

For the definition and the basic properties of the ring of Witt vectors we refer to \cite[\textsection1]{H}. We briefly recall the notions in this section. 

A  subset $S\subset \mathbb{N}=\{1,2,\dots\}$ is called a \emph{truncation set} if $n\in S$ implies 
that all positive divisors of $n$ are contained in $S$. For a truncation set $S$ and $n\in S$, we define
$
S/n:=\{s\in S\mid sn\in S\}.
$

Let $A$ be a commutative ring. For all truncation sets $S$ we have the ring of Witt vectors 
$
\W_S(A)
$
at our disposal. The ghost map is the functorial ring homomorphism 
$$  gh=(gh_n)_{n\in S}:\W_S(A)\xr{} \prod_{n\in S}A,\quad
  gh_n((a_s)_{s\in S}):=\sum_{d\mid n} d\cdot a_d^{n/d}.$$
It is injective provided that $A$ is $\Z$-torsion-free.

For all positive integers $n$, there is a functorial morphism of rings 
$$
F_n:\W_S(A)\xr{}\W_{S/n}(A),
$$
called the \emph{Frobenius}. Moreover there is a functorial morphism of $\W_S(A)$-modules, the \emph{Verschiebung},  
\begin{align*}
V_n: \W_{S/n}(A)&\xr{} \W_S(A),  
\end{align*}
where the source is a $\W_S(A)$-module via $F_n$. 
For all coprime positive integers $n,m \in \mathbb{N}$ we have 
$$
F_n\circ V_n=n, \quad F_n\circ V_m = V_m \circ F_n  \qquad \quad ((m,n)=1).
$$

We have a multiplicative Teichm\"uller map
$$
[-]:A\xr{} \W_S(A), \quad a\mapsto [a]:=(a,0,0,\dots)\in \W_S(A),
$$
and if $S$ is finite then every element $a\in \W_S(A)$ can be written as 
$$
a=\sum_{s\in S}V_s([a_s])
$$ 
with unique elements $(a_s)_{s\in S}$ in $A$.

Let $T\subset A$ be a multiplicative set and suppose that $S$
is a finite truncation set. We can consider $T$ via the Teichm\"uller map
as multiplicative set in $\W_S(A)$. Then the natural ring homomorphism
$$
T^{-1}\W_S(A)\xr{} \W_S(T^{-1}A).
$$  
is an isomorphism. If $T\subset \Z$ is a multiplicative set then 
$$
\W_S(A)\otimes_{\Z}T^{-1}\Z \xr{} \W_S(T^{-1}A)
$$
is an isomorphism. 

Let $S$ be a truncation set, and let $n$ be a positive integer; set $T:=S\backslash \{s\in S; n\mid s\}$. Then $T$  is a truncation set and  
we have a short exact sequence of $\W_S(A)$-modules: 
\begin{equation}\label{equation-short-exact-seq-S/n-S-T}
0\xr{} \W_{S/n}(A)\xr{V_n} \W_S(A)\xr{R^S_{T}} \W_T(A)\xr{} 0.
\end{equation}

\begin{example}
  We have 
$
\W_S(\Z)=\prod_{n\in S} \Z\cdot V_n(1),
$
and the product is given by $V_m(1)\cdot V_n(1)=c\cdot V_{mn/c}(1)$, where $c=(m,n)$ is the greatest common divisor \cite[Proposition~1.6]{H}.
\end{example}
 



\subsubsection{}
The following theorem will be very useful throughout the paper. 

\begin{thm}\label{thm-van-der-Kallen-Borger}(Borger-van der Kallen)
Let $S$ be a finite truncation set, and let $n$ be a positive integer. 
Let $\rho:A\xr{} B$ be an \'etale ring homomorphism.
The following hold.
\begin{enumerate}
\item The induced ring homomorphism $\W_S(A)\xr{}\W_S(B)$ is \'etale.
\item The morphism 
$$
\W_S(B)\otimes_{\W_S(A),F_n}\W_{S/n}(A)\xr{} \W_{S/n}(B), \quad b\otimes a\mapsto F_n(b)\cdot \W_{S/n}(\rho)(a),
$$
is an isomorphism.
\end{enumerate}
\end{thm}

The references for this theorem are \cite[Theorem~B]{Borger3} \cite[Corollary~15.4]{Borger4}  and \cite[Theorem~2.4]{vdK} (cf. \cite[Theorem~1.22]{H}).

By using Theorem~\ref{thm-van-der-Kallen-Borger}, the exact sequence (\ref{equation-short-exact-seq-S/n-S-T}), and 
induction on the length of $S$, we easily obtain the following corollary.

\begin{corollary}\label{corollary-van-der-Kallen-Borger}
  Let $\rho:A\xr{} B$ be an \'etale ring homomorphism. Let $S$ be a finite 
truncation set. 
  \begin{itemize}
  \item [(i)]   For an inclusion of truncation sets $T\subset S$, the map
$$
\W_S(B)\otimes_{\W_S(A)}\W_T(A)\xr{}\W_T(B)
$$ 
is an isomorphism.
\item [(ii)] Let $n$ be a positive integer. For any $A$-algebra $C$, the natural ring homomorphism 
$$
\W_{S/n}(C)\otimes_{F_n,\W_S(A)}\W_{S}(B) \xr{} \W_{S/n}(C\otimes_A B), \quad c\otimes b\mapsto c\cdot F_n(b)
$$
is an isomorphism.
  \end{itemize}
\end{corollary}

\begin{notation}\label{notation-Wn}
If a prime $p$ has been fixed then we set  $W_n:=\W_{\{1,p,p^2,\dots,p^{n-1}\}}$.  
\end{notation}
\subsubsection{}\label{section-epsilon-decomposition}
Let $p$ be a prime. Let $R$ be a $\Z_{(p)}$-algebra.  
Since all primes different from $p$ are invertible in $R$, 
the same holds in $\W_S(R)$.  The category of $\W_S(R)$-modules, for a finite truncation set $S$, factors in the following way. Set 
$$
\epsilon_{1,S}:=\prod_{\substack{\text{primes $\ell\neq p$}\\ S/\ell\neq \emptyset}} (1-\frac{1}{\ell}V_{\ell}(1))\in \W_S(R),  
$$
and $\epsilon_{n,S}:=\frac{1}{n}V_n\left(\epsilon_{1,S/n}\right)$ for all positive integers $n$ with $(n,p)=1$.
Of course, if $S/n=\emptyset$ then $\epsilon_{S,n}=0$.
In the following we will simply write $\epsilon_n$ for $\epsilon_{n,S}$.
For all positive integers $n\neq n'$ with $(n,p)=1=(n',p)$ the equalities
$$
\epsilon_{n}^2=\epsilon_{n}, \qquad \epsilon_{n}\epsilon_{n'}=0,
$$
hold. Moreover, if $(m,p)=1=(n,p)$ then
$$
F_m(\epsilon_n)=\begin{cases} \epsilon_{n/m} & \text{if $m\mid n$,} \\
                              0  & \text{if $m\nmid n$}.\end{cases}
$$
Since $\sum_{(n,p)=1} \epsilon_n=1$ we obtain a decomposition of rings
\begin{equation}\label{equation-epsilon-decomposition-Witt-vectors}
\W_S(R)=\prod_{n\geq 1, (n,p)=1} \epsilon_n \W_S(R).  
\end{equation}

\begin{notation}\label{notation-S-p}
For a finite truncation set $S$ we denote by $S_p$ the elements in $S$ 
that are $p$-powers, that is 
$
S_p=S\cap \{p^i\mid i\geq 0\}.
$   
\end{notation}

The map $$R_{(S/n)_p}^{S/n}\circ F_n:\W_S(R)\xr{} \W_{(S/n)_p}(R)$$ induces an 
isomorphism $\epsilon_n \W_S(R) \cong \W_{(S/n)_p}(R)$. Thus 
$$
M\mapsto \bigoplus_{n\geq 1,(n,p)=1} \epsilon_n M 
$$
defines an equivalence of categories
\begin{equation}\label{equation-epsilon-decomposition}
(\text{$\W_S(R)$-modules})\xr{\cong} \prod_{n\geq 1, (n,p)=1}(\text{$\W_{(S/n)_p}(R)$-modules}).  
\end{equation}

\subsubsection{} The following two lemmas are concerned with maximal ideals in $\W_S(R)$. 
\begin{lemma}\label{lemma-points-of-WSR}
  Let $R$ be a ring. Let $S$ be a finite truncation set. For every maximal ideal $\mathfrak{m}\subset \W_S(R)$
there exists a maximal ideal $\mathfrak{p}\subset R$ such that $ \W_S(R) \xr{} \W_S(R)/ \mathfrak{m}$ factors through $\W_S(R_{\mathfrak{p}})$.
\begin{proof}
Set $k=\W_S(R)/ \mathfrak{m}$, we distinguish two cases:
\begin{enumerate}
\item $k$ has characteristic $0$,
\item $k$ has characteristic $p>0$.
\end{enumerate}
In the first case we can factor 
$$
\W_S(R) \xr{} \W_S(R)\otimes_{\Z}\Q\xr{=} \W_S(R\otimes_{\Z}\Q)\xr{} k.
$$
Since $\W_S(R\otimes_{\Z}\Q)\xr{gh,\cong} \prod_{s\in S}R\otimes \Q$, the claim follows.

Suppose now that $k$ has characteristic $p>0$. We have a factorization 
$$\W_S(R) \xr{} \W_S(R)\otimes_{\Z}\Z_{(p)}\xr{=} \W_S(R\otimes_{\Z}\Z_{(p)})\xr{} k.$$
By decomposing 
\begin{multline*}
\W_S(R\otimes \Z_{(p)}) \xr{=} \prod_{n\geq 1,(n,p)=1} \epsilon_n \W_S(R\otimes \Z_{(p)}) \\ \xr{\cong,\prod_n R^{S/n}_{(S/n)_p}\circ F_n} \prod_{n\geq 1,(n,p)=1} \W_{(S/n)_p}(R\otimes \Z_{(p)}),
\end{multline*}
we can reduce to the case where $S$ consists only of $p$-powers. 
Finally, $V_{p}(a)^2=pV_{p}(a^2)$, for all $a\in \W_{S/p}(R\otimes \Z_{(p)})$, hence $V_{p}(a)$ maps to zero in $k$. Therefore $\W_S(R\otimes \Z_{(p)})\xr{} k$ factors through $\W_{S}(R\otimes \Z_{(p)})\xr{} \W_{\{1\}}(R\otimes \Z_{(p)})=R\otimes \Z_{(p)}\xr{\rho} k$. In this case we can take 
$\mathfrak{p}=\ker(R\xr{} R\otimes \Z_{(p)} \xr{\rho} k).$
\end{proof}
\end{lemma}

\begin{lemma}\label{lemma-WSRp-local-more-general}
Let $p$ be a prime. Let $R$ be a ring such that every  maximal ideal $\mf{p}$ satisfies ${\rm char}(R/\mf{p})=p>0$.
Let $S$ be a $p$-typical finite truncation set. 
Then every maximal ideal $\mf{m}$ of $\W_S(R)$ is of the form
$
\ker(\W_S(R)\xr{R^S_{\{1\}}} R\xr{} R/\mf{p}),
$
for a unique maximal ideal $\mf{p}$ of $R$.
\begin{proof}
 Let $\mf{m}$ be a maximal ideal of $\W_S(R)$, set $k=\W_S(R)/\mf{m}$. 
We claim that ${\rm char}(k)=p$. Suppose that ${\rm char}(k)\neq p$. From the 
commutative diagram 
$$
\xymatrix
{
\W_S(R) \ar[r]\ar[d]^{gh}
&
\W_S(R)\otimes \Z[p^{-1}] \ar[r] \ar[d]^{gh}_{\cong}
&
k
\\
\prod_{s\in S} R\ar[r]
&
\prod_{s\in S} R\otimes \Z[p^{-1}] \ar[ru]
&
}
$$ 
we conclude that there is a factorization $\W_S(R)\xr{gh_i}R\xr{} k$, but there are no epimorphism $R\xr{} k$ to a field of characteristic $\neq p$.

Thus we may suppose that ${\rm char}(k)=p$. Because $V_p(a)^2=pV_p(a^2)$ for 
all $a\in \W_{S/p}(R)$, we obtain a factorization $\W_S(R)\xr{R^S_{\{1\}}}R\xr{} k$, which defines $\mf{p}:=\ker(R\xr{} k)$.
\end{proof}
\end{lemma}

\subsection{Relative de Rham-Witt complex}

For every commutative ring $A$ we have the absolute de Rham-Witt complex $$ S\mapsto \W_S\Omega^*_{A}$$ constructed by Hesselholt \cite{H}, at our disposal.
The absolute de Rham-Witt complex is the initial object in the category of Witt complexes \cite[\textsection4]{H}. In this section we will define
the relative version, which is studied in this paper. 

\begin{definition}
Let $A$ be an $R$-algebra.  Let $S$ be a truncation set and $q\geq 0$. We define 
$$
\mathbb{W}_S\Omega^{q}_{A/R}= \varprojlim_{\substack{T\subset S\\ \text{$T$ finite}}} \mathbb{W}_T\Omega^{q}_{A}/\left( \mathbb{W}_T\Omega^{1}_{R} \cdot \mathbb{W}_T\Omega^{q-1}_{A}\right) 
$$
For $q=0$, the definition means $\mathbb{W}_S\Omega^{0}_{A/R}=\W_S(A)$.
\end{definition}
We get an induced anti-symmetric graded algebra structure on $\mathbb{W}_S\Omega^{*}_{A/R}$, that is, $\omega_1\cdot \omega_2=(-1)^{\deg(\omega_1)\deg(\omega_2)}\omega_2\cdot \omega_1$.

Recall that by construction of $\W_S\Omega^*_A$, there is, for all finite truncation sets $S$, a surjective morphism of graded $\W_S(A)$-algebras 
\begin{equation}\label{equation-tensor-to-absolute-dRW} 
  \pi:T^*_{\W_S(A)} \Omega^1_{\W_S(A)} \xr{} \W_S\Omega^*_{A},
\end{equation}
such that $\pi(da)=da$ for all $a\in \W_S(A)$.

\begin{lemma}\label{lemma-morph-de-Rham-to-de-Rham-Witt}
Let $S$ be a finite truncation set.  
\begin{enumerate}
  \item The morphism \eqref{equation-tensor-to-absolute-dRW} induces a surjective morphism of anti-symmetric graded algebras 
    \begin{equation}\label{equation-dR-to-absolute-dRW}
      \pi:\Omega^*_{\W_S(A)/\W_S(R)} \xr{} \W_S\Omega^*_{A/R},
    \end{equation}
   which by abuse of notation is called $\pi$ again.
  \item $\W_S\Omega^*_{A/R}$ is a differential graded algebra and (\ref{equation-dR-to-absolute-dRW}) is compatible with the differential. 
  \end{enumerate}
\begin{proof}
  For (1). This follows from $\pi(da\otimes da)\in d\log[-1]\cdot \W_S\Omega^1_A$ \cite[\textsection3]{H} and $d\log[-1]\in \W_S\Omega^1_R$. 

For (2). The differential $d:\W_S\Omega^*_{A/R}\xr{} \W_S\Omega^*_{A/R}$ is well-defined, because $\W_S\Omega^*_R$ is generated by $\W_S\Omega^1_R$. It satisfies $d\circ d=0$, because $d\log[-1]\in \W_S\Omega^1_R$. The compatibility of $\pi$ with $d$ follows from $\pi(da)=da$ for all $a\in \W_S(A)$.  
\end{proof}
\end{lemma}

\subsubsection{} 
Induced from the absolute de Rham-Witt complex, we obtain for all positive integers $n$:
\begin{align}
F_n:&\mathbb{W}_S\Omega^{q}_{A/R}\xr{} \mathbb{W}_{S/n}\Omega^{q}_{A/R}\label{equation-F-relative-over-Z},\\
V_n:&\mathbb{W}_{S/n}\Omega^{q}_{A/R} \xr{} \mathbb{W}_S\Omega^{q}_{A/R}, \label{equation-V-relative-over-Z}
\end{align}  
and $S\mapsto \mathbb{W}_S\Omega^{*}_{A/R}$ forms a Witt complex. Note that, computed  in the absolute de Rham-Witt complex,
we have
\begin{align*}
  V_n(da\cdot \omega)=V_n(F_ndV_n(a)\cdot \omega)=dV_n(a)\cdot V_n(\omega),
\end{align*}
hence $V_n(\W_{S/n}\Omega^1_R\cdot \W_{S/n}\Omega^{q-1}_A)\subset \W_S\Omega^1_R\cdot \W_S\Omega^{q-1}_A$.
  
The following equalities hold for the maps (\ref{equation-F-relative-over-Z}), (\ref{equation-V-relative-over-Z}):
$$
V_nF_nd=dV_nF_n, \quad dV_nd=0.
$$

\begin{proposition}\label{proposition-relative-de-Rham-Witt-complex-initial}
The Witt complex $S\mapsto \W_S\Omega^*_{A/R}$ is the initial object in the
category of Witt complexes over $A$ with $\W(R)$-linear differential. 
\begin{proof}
  Let $S\mapsto E_S^*$ be a Witt complex over $A$ with $\W(R)$-linear differential, that is, $d(a\omega)=ad(\omega)$ for $a\in \W_S(R)$ and 
$\omega\in E_S^*$. We only need to show that the canonical morphism 
$$
[S\mapsto \W_S\Omega^*_{A}]\xr{} [S\mapsto E_S^*]
$$
factors through $[S\mapsto \W_S\Omega^*_{A/R}]$. It is enough to check this for finite truncation sets. Because $\pi$ (\ref{equation-tensor-to-absolute-dRW}) is surjective, we conclude that $\mathbb{W}_S\Omega^1_{R}$ is generated by elements of the form $da$ with $a\in \W_S(R)$, which implies the claim. 
\end{proof}
\end{proposition}

As a corollary we obtain the following statement.

\begin{corollary}\label{corollary-de-Rham-Witt-complex-localization}
Let $A$ be an $R$-algebra, let $p$ be a prime, and set $R':=R\otimes_{\Z}\Z_{(p)},A':=A\otimes_{\Z}\Z_{(p)}$. 
There is a unique isomorphism 
$$
[S\mapsto \W_S\Omega^*_{A'/R'}]\xr{} [S\mapsto \varprojlim_{\substack{T\subset S\\\text{$T$ finite}}}\W_T\Omega^*_{A/R}\otimes_{\Z}\Z_{(p)}]
$$
of Witt complexes over $A'$.
\end{corollary}

\begin{proposition}\label{proposition-epsilon-decomposion-deRham-Witt}
Let $R$ be a $\Z_{(p)}$-algebra and let $A,B$ be $R$-algebras. Let $S$ be a finite truncation set. 
\begin{enumerate}
\item Via the equivalence from \eqref{equation-epsilon-decomposition}
we have 
\begin{equation}\label{equation-epsilon-decomposion-deRham-Witt}
\W_S\Omega^*_{A/R}\mapsto \bigoplus_{n\geq 1, (n,p)=1} \W_{(S/n)_p}\Omega^*_{A/R}.  
\end{equation}
\item For a morphism $f:A\xr{} B$ the induced morphism $f_S:\W_S\Omega^*_{A/R}
\xr{} \W_S\Omega^*_{B/R}$ maps to
$$
f_S\mapsto \bigoplus_{n\geq 1, (n,p)=1} f_{(S/n)_p}
$$
via the equivalence from \eqref{equation-epsilon-decomposition}.
\end{enumerate}
\begin{proof}
For (1). The claim follows from \cite[Proposition~1.2.5]{HM}. In the notation of loc.~cit.~the right hand side (\ref{equation-epsilon-decomposion-deRham-Witt})
equals $i_{!}i^*\W\Omega^*_{A/R}$, and $i^*,i_{!}$ preserve initial objects, since
both functors admit a right adjoint. 

For (2). Follows immediately from the construction in (1).
\end{proof}
\end{proposition}

\begin{proposition}\label{proposition-comparison-Langer-Zink}
Let $R$ be a $\Z_{(p)}$-algebra,  let $A$ be an $R$-algebra. Then 
$$
n\mapsto \W_{\{1,p,\dots,p^{n-1}\}}\Omega^*_{A/R}
$$
is the relative de Rham-Witt complex $n\mapsto W_{n}\Omega^*_{A/R}$ defined by Langer and Zink \cite{LZ}.
\begin{proof}
  We have a restriction functor 
  \begin{multline*}
i^*:\text{(Witt systems over $A$ with $\W(R)$-linear differential)} \xr{}\\
\text{($F$-$V$-procomplexes over the $R$-algebra $A$),}   
  \end{multline*}
where we use the definition of \cite[\textsection4]{H} for the source category
and the definition of \cite[Introduction]{LZ} for the target category.
The functor $i^*$ admits a right adjoint functor $i_!$ defined in \cite[\textsection1.2]{HM}. Therefore $i^*([S\mapsto \W_S\Omega^*_{A/R}])$ is the 
initial object in the category of $F$-$V$-procomplexes as is the relative
de Rham-Witt complex constructed by Langer and Zink \cite{LZ}.
\end{proof}
\end{proposition}

\subsubsection{}
Let $S$ be a finite truncation set. 
Let $A\xr{} B$ be an \'etale morphism of $R$-algebras. For all $q\geq 0$ the induced morphism of $\W_S(B)$-modules
\begin{equation}\label{equation-relative-big-etale-base-change}
\W_S(B)\otimes_{\W_S(A)} \W_S\Omega^q_{A/R}\xr{\cong} \W_S\Omega^q_{B/R}  
\end{equation}
is an isomorphism. Indeed, this follows immediately from the analogous fact for the absolute de Rham-Witt complex \cite[Theorem~C]{H}.

\begin{lemma}
  Let $R'\xr{} R$ be an \'etale ring homomorphism. Let $A$ be an $R$-algebra. Then, for all truncation sets $S$,
$$
\W_S\Omega^*_{A/R'}\xr{} \W_S\Omega^*_{A/R}
$$
is an isomorphism. 
\begin{proof}
  We may assume that $S$ is finite. The assertion follows from
  \begin{align*}
\W_S\Omega^1_{R'}\otimes_{\W_S(R')}\W_S(A) &\xr{=}\W_S\Omega^1_{R'}\otimes_{\W_S(R')} \W_S(R) \otimes_{\W_S(R)} \W_S(A)\\
      &\xr{\cong} \W_S\Omega^1_R\otimes_{\W_S(R)} \W_S(A).    
  \end{align*}
\end{proof}
\end{lemma}

\subsubsection{}
For every truncation set $S$ we have a functor
$$
\W_S:\text{(Schemes)}\xr{} \text{(Schemes)}, \quad X\mapsto \W_S(X).
$$
This functor has been studied by Borger \cite{Borger4}, our notation differs slightly: the notation is $W^*$ in \cite{Borger4}. 

For an affine scheme $U=\Spec(A)$, we have $\W_S(U)=\Spec(\W_S(A))$.
If $X$ is separated and $(U_i)_{i\in I}$ is an affine covering of $X$, then $\W_S(X)$ is obtained by gluing $\W_S(U_i)$  
along $\W_S(U_i\times_X U_j)$. In particular, $(\W_S(U_i))_{i\in I}$ is an affine covering of $\W_S(X)$.
The functor is extended to non-separated schemes in the usual way. 

If $T\subset S$ is an inclusion of finite truncation sets then 
$$\imath_{T,S}:\W_T(X)\xr{} \W_S(X)$$ 
is a closed immersion and functorial in $X$.

\subsubsection{}
If $X$ is an $R$-scheme then we can glue in the same way a quasi-coherent sheaf $\W_S\Omega^q_{X/R}$. Indeed, let us suppose that $X$ is separated. Let
$(\Spec(A_i))_{i\in I}$ be an affine covering and set $\Spec(A_{ij})=\Spec(A_i)\times_{X}\Spec(A_j)$. For every $i$, the $\W_S(A_i)$-module
$\W_S\Omega^q_{A_i/R}$ defines a quasi-coherent sheaf $\W_S\Omega^q_{\Spec(A_i)/R}$ on $\W_S(\Spec(A_i))$. Since
$$
\Gamma(\W_S(\Spec(A_{ij}))),\W_S\Omega^q_{\Spec(A_i)/R}) = \W_S\Omega^q_{A_i/R}\otimes_{\W_S(A_i)}\W_S(A_{ij}) = \W_S\Omega^q_{A_{ij}/R},
$$  
by using  (\ref{equation-relative-big-etale-base-change}), we can glue to a quasi-coherent sheaf $\W_S\Omega_{X/R}$ on $\W_S(X)$. 
Independence of the covering and $\jmath^*\W_S\Omega^q_{X/R}=\W_S\Omega^q_{U/R}$, for every open $\jmath:U\xr{} X$, can be checked. 

\subsubsection{}
\label{subsubsection-coherent}
If $\W_S(X)\xr{} \W_S(\Spec(R))$ is of finite type and $\W_S(X)$ is noetherian, then $\W_S\Omega^q_{X/R}$ is coherent. Indeed, 
we have a surjective morphism $\Omega^j_{\W_S(X)/\W_S(R)}\xr{} \W_S\Omega^j_{X/R}$ and the assumptions imply that $\Omega^j_{\W_S(X)/\W_S(R)}$ is coherent. 

\subsubsection{}
If $f:X\xr{} Y$ is a morphism of $R$-schemes then we get $$\W_S\Omega^q_{Y/R}\xr{} \W_S(f)_*\W_S\Omega^q_{X/R}.$$ 
For an inclusion of truncation  sets $T\subset S$, we obtain
$$
\W_S \Omega^q_{X/R}\xr{} \imath_{T,S*}\W_T\Omega^q_{X/R}.
$$
The following diagram is commutative: 
$$
\xymatrix
{
\W_S\Omega^q_{Y/R} \ar[rr]\ar[d]
&
&
\W_S(f)_*\W_S\Omega^q_{X/R} \ar[d]
\\
\imath_{T,S*}\W_T\Omega^q_{Y/R} \ar[r]
&
\imath_{T,S*}\W_T(f)_*\W_T\Omega^q_{X/R} \ar[r]^{=}
&
\W_S(f)_*\imath_{T,S*}\W_T\Omega^q_{X/R}.
}
$$
The differential, the Frobenius and the Verschiebung operations are defined in the evident way:
\begin{align*}
&d:\W_{S}\Omega^q_{X/R} \xr{} \W_{S}\Omega^{q+1}_{X/R},\\
&F_n: \W_{S}\Omega^q_{X/R} \xr{} \imath_{S/n,S*}\W_{S/n}\Omega^q_{X/R}, \\
&V_n: \imath_{S/n,S*}\W_{S/n}\Omega^q_{X/R} \xr{} \W_{S}\Omega^q_{X/R}.
\end{align*}

\begin{definition}\label{definition-dRW-coh-new}
  Let $X$ be an $R$-scheme, let $S$ be a finite truncation set. We define 
$$
H^i_{dRW}(X/\W_S(R)):=H^i(\W_S(X),\W_S\Omega^*_{X/R}),
$$
where the right hand side is the hypercohomology for the Zariski topology.
\end{definition}

\subsubsection{}\label{section-phi-beta-de-Rham-Witt}
Note that $F_n$ and $V_n$ are not morphisms of complexes. 
For all positive integers $n$ and all finite truncation sets  we set 
\begin{equation}\label{equation-definition-phin}
\phi_n=n^qF_n:\W_S\Omega^q_{X/R}\xr{}   \imath_{S/n,S*}\W_{S/n}\Omega^q_{X/R},  
\end{equation}
to get a morphism of complexes 
$$
\W_S\Omega^*_{X/R}\xr{\phi_n}  \imath_{S/n,S*}\W_{S/n}\Omega^*_{X/R}.
$$
Suppose that $X$ is smooth over $R$ of relative dimension $d$. Then we set
$$
\beta_n=n^{d-q}V_n:\imath_{S/n,S*}\W_{S/n}\Omega^q_{X/R}\xr{}  \W_{S}\Omega^q_{X/R}
$$
(we will prove $\W_S\Omega^q_{X/R}=0$ if $q>d$ in Proposition \ref{proposition-WSOmega-torsionfree}(ii)). We obtain a morphism of complexes 
$$
\beta_n:\imath_{S/n,S*}\W_{S/n}\Omega^*_{X/R}\xr{}  \W_{S}\Omega^*_{X/R},
$$
satisfying the equalities:
\begin{align*}
&\phi_n\circ \beta_n=n^{d+1}, \\
&\beta_n(\lambda \cdot \phi_n(x))=n^{d}V_n(\lambda)\cdot x \quad \text{for all $x\in \W_S\Omega^*_{X/R}$ and $\lambda\in \imath_{S/n,S*}\W_{S/n}\Omega^*_{X/R}$.}  
\end{align*}
In Section \ref{section-values} we will study the $\{\phi_n\}_{n\geq 1}$ operations induced on the de Rham-Witt cohomology.

\subsubsection{}\label{subsubsection-can-be-computed-by-Cech-covering}
Note that the Hodge to de Rham spectral sequence and the quasi-coherence of $\W_S\Omega^q_{X/R}$ imply the following fact. 
Assume $X$ is separated and $\W_S(X)$ is a noetherian scheme. 
Let $(U_i)$ be an open affine covering for $X$, we denote by $\mathfrak{U}=(\W_S(U_i))$ the induced covering of $\W_S(X)$. Then we can compute
$H^i_{dRW}(X/\W_S(R))$ by using the \v{C}ech complex for $\mathfrak{U}$:
$$
H^i(C(\mathfrak{U},\W_S\Omega^*_{X/R}))\xr{\cong} H^i_{dRW}(X/\W_S(R)).
$$
In the derived category we have a quasi-isomorphism:
$$
C(\mathfrak{U},\W_S\Omega^*_{X/R})\xr{\text{q-iso}} R\Gamma(\W_S\Omega^*_{X/R}).
$$

\begin{proposition}\label{proposition-WSOmega-torsionfree}
Let $R$ be a flat $\Z$-algebra. Let $X$ be a smooth $R$-scheme. Let $S$ be a finite truncation set. 
\begin{itemize}
\item[(i)] For all non-negative integers $q$, $\W_S\Omega^q_{X/R}$ is $\Z$-torsion-free, that is, 
multiplication by a non-zero integer is injective. 
\item[(ii)] Let $d$ be the relative dimension of $X/R$. Then $\W_S\Omega^q_{X/R}=0$ for all $q>d$.
\end{itemize}
\begin{proof}
  For (i) it suffices to prove that $\W_S\Omega^q_{X/R}\otimes \Z_{(p)}=\W_S\Omega^q_{X'/R'}$ is $p$-torsion-free for all primes $p$,
 where $X'=X\otimes_{\Z} \Z_{(p)}$  and $R'=R\otimes_{\Z} \Z_{(p)}$. For (ii) it suffices to show that $\W_S\Omega^q_{X'/R'}$ vanishes.

Via the decomposition \ref{equation-epsilon-decomposion-deRham-Witt} we may suppose that $S=\{1,p,\dots,p^{n-1}\}$. 
Certainly we may assume that $X'=\Spec(B)$ and that there exists an \'etale ring homomorphism $R'[x_1,\dots,x_d]\xr{} B$. By using 
\eqref{equation-relative-big-etale-base-change} we are reduced to the case 
$B=R'[x_1,\dots,x_d]$. The claim follows in this case from the explicit description of the de Rham-Witt complex 
in \cite[\textsection2]{LZ}, more precisely \cite[Proposition~2.17]{LZ}. 
\end{proof}
\end{proposition}

\subsection{Finiteness}

\begin{proposition}\label{proposition-finiteness}
Let $R$ be a flat and finitely generated $\Z$-algebra. Let $X$ be a flat and proper scheme of relative dimension $d$ over $R$. 
Let $S$ be a finite truncation set. The following hold.
\begin{itemize}
\item [(i)] For all non-negative integers $i,j$ the cohomology group
$
H^i(\W_S(X),\W_S\Omega^j_{X/R})
$
is a finitely generated $\W_S(R)$-module. 
\item [(ii)] For all $i> d$ and $j\geq 0$, we have $H^i(\W_S(X),\W_S\Omega^j_{X/R})=0$.
\item [(iii)] For all $i$, the de Rham-Witt cohomology $H^i_{dRW}(X/\W_S(R))$ (Definition ~\ref{definition-dRW-coh-new}) is a finitely generated 
$\W_S(R)$-module. 
\item [(iv)] Suppose $X/R$ is smooth. Then $H^i_{dRW}(X/\W_S(R))=0$ for all $i>2d$.
\end{itemize}
\begin{proof}
For (i).  We denote by $f:X\xr{} \Spec(R)$ the structure morphism. The scheme $\W_S(X)$
is noetherian, because it is of finite type over $\Spec(\Z)$. 
By \cite[Proposition~16.13]{Borger4} the induced morphism $\W_S(f):\W_S(X)\xr{}\W_S(R)$ is proper. 
Moreover, $\W_S\Omega^j_{X/R}$ defines a coherent sheaf on $\W_S(X)$ (see \ref{subsubsection-coherent}). 

For (ii). The fibers of $\W_S(f)$ at closed points of $\Spec(\W_S(R))$ have 
dimension $d$. In fact, as topological spaces they are disjoint unions
of the corresponding fibers of $f$. This implies the claim.

For (iii). Follows from (i) via the Hodge to de Rham spectral sequence.

For (iv). Again this follows from the Hodge to de Rham spectral sequence, statement (ii), and Proposition \ref{proposition-WSOmega-torsionfree}(ii). 
\end{proof}
\end{proposition}




\section{De Rham-Witt cohomology}

\subsection{Reduction modulo an ideal}

\subsubsection{}
Recall that $W_n=\W_{\{1,p,\dots,p^{n-1}\}}$ whenever a prime $p$ has been fixed (Notation~\ref{notation-Wn}).
The goal of this section is to prove the following theorem.

\begin{thm}\label{thm-comparison-isom}
  Let $R$ be a flat $\Z_{(p)}$-algebra, let $B$ be a smooth $R$-algebra, and let $n$ be positive integers. 
Let $I\subset R$ be an ideal such that $p^m\in I$ for some $m$. 
Choose a $W_n(R)$-free resolution 
$$
T:=\dots \xr{} T^{-2}\xr{} T^{-1}\xr{} T^{0}
$$ 
of $W_n(R/I)$.
There exists a functorial quasi-isomorphism of complexes of $W_n(R)$-modules
\begin{equation}\label{equation-comp-isomorphism-with-T}
W_n\Omega^*_{B/R}\otimes_{W_n(R)} T\xr{} W_n\Omega^*_{(B/IB)/(R/I)}.  
\end{equation}
In particular, we obtain an isomorphism
\begin{equation}\label{equation-comp-isomorphism}
W_n\Omega^*_{B/R}\otimes^{\mathbb{L}}_{W_n(R)} W_n(R/I)\xr{\cong} W_n\Omega^*_{(B/IB)/(R/I)},  
\end{equation}
in the derived category of $W_n(R)$-modules.
\end{thm}

More precisely, functoriality means  that for any morphism $A\xr{} B$ of smooth $R$-algebras, the diagram 
$$
\xymatrix
{
W_n\Omega^*_{B/R}\otimes_{W_n(R)} T \ar[r]
&
W_n\Omega^*_{(B/IB)/(R/I)} 
\\
W_n\Omega^*_{A/R}\otimes_{W_n(R)} T \ar[r]\ar[u]
&
W_n\Omega^*_{(A/IA)/(R/I)} \ar[u]
}
$$  
is commutative. 

\begin{remark}
The proof of Theorem \ref{thm-comparison-isom} does not go
beyond the methods of \cite{LZ}, so that the theorem may be well-known but
we couldn't provide a reference.
\end{remark}

\begin{proof}[Proof of Theorem \ref{thm-comparison-isom}]
  We define the morphism (\ref{equation-comp-isomorphism-with-T}) by 
$$
W_n\Omega^*_{B/R}\otimes_{W_n(R)} T \xr{} W_n\Omega^*_{B/R}\otimes_{W_n(R)} W_n(R/I) \xr{} W_n\Omega^*_{(B/I)/(R/I)},
$$
so that the functoriality of (\ref{equation-comp-isomorphism-with-T}) is obvious.

\emph{1.Step:} The first step is the reduction to $B=R[x_1,\dots,x_d]$. 
We can use the \v{C}ech complex (see \ref{subsubsection-can-be-computed-by-Cech-covering}) 
in order to reduce to the case where there exists an \'etale morphism $A=R[x_1,\dots,x_d]\xr{} B$.

Note that $p^{nm}=0$ in $W_n(R/I)$.
Since $W_n\Omega^*_{B/R}$ is $p$-torsion-free (Proposition ~\ref{proposition-WSOmega-torsionfree}), we see that 
\begin{equation}\label{equation-mod-pnm} 
W_n\Omega^*_{B/R}\otimes^{\mathbb{L}}_{W_n(R)} W_n(R/I)\xr{} 
W_n\Omega^*_{B/R}/p^{nm}\otimes^{\mathbb{L}}_{W_n(R)/p^{nm}} W_n(R/I) 
\end{equation}
is a quasi-isomorphism. Clearly, morphism (\ref{equation-comp-isomorphism})
factors through (\ref{equation-mod-pnm}). It will be easier to work modulo $p^{nm}$,
because $dF^{nm}_{p}=p^{nm}F^{nm}_pd$ vanishes modulo $p^{nm}$.

Set $c=nm+n$, we claim that 
\begin{align}\label{align-from-A-to-B-modulo-pnm}
\left(W_{c}(B)/p^{nm}\otimes_{W_{c}(A)/p^{nm}} W_n\Omega^*_{A/R}/p^{nm},
id\otimes d\right)  &\xr{} (W_n\Omega^*_{B/R}/p^{nm},d) \\
b\otimes \omega &\mapsto F_p^{nm}(b)\cdot \omega, \nonumber
\end{align}
is an isomorphism of complexes. Note that $W_c(A)$ acts on $W_n\Omega^*_{A/\Z}/p^{nm}$ via $W_c(A)\xr{F^{nm}_p} W_n(A)$, and therefore (\ref{align-from-A-to-B-modulo-pnm})
is a morphism of complexes. 
Theorem \ref{thm-van-der-Kallen-Borger} implies 
that 
$$
W_{c}(B)\otimes_{W_c(A)} M\xr{\cong} W_{n}(B)\otimes_{W_n(A)} M, \quad b\otimes m\mapsto F^{nm}_p(b)\otimes m,
$$
is an isomorphism for all $W_n(A)$-modules $M$.  Thus the claim follows from (\ref{equation-relative-big-etale-base-change}).

On the other hand, Corollary \ref{corollary-van-der-Kallen-Borger} shows that for every $W_n(A/I)$-module $M$ the map 
$$
W_{c}(B)/p^{nm}\otimes_{W_{c}(A)/p^{nm}}M \xr{} W_n(B/I)\otimes_{W_n(A/I)} M, \quad b\otimes m\mapsto F_{p}^{nm}(b)\otimes m,
$$
is an isomorphism. This yields an isomorphism of complexes 
$$
\left(W_{c}(B)/p^{nm}\otimes_{W_{c}(A)/p^{nm}} W_n\Omega^*_{(A/IA)/(R/I)},
id\otimes d\right)  \xr{} (W_n\Omega^*_{(B/IB)/(R/I)},d).
$$
Finally, since $W_{c}(B)/p^{nm}$ is \'etale over $W_{c}(A)/p^{nm}$, we are reduced to proving that 
$$
W_n\Omega^*_{A/R}/p^{nm}\otimes^{\mathbb{L}}_{W_n(R)/p^{nm}} W_n(R/I)
\xr{} W_n\Omega^*_{(A/IA)/(R/I)}
$$
is a quasi-isomorphism. 

\emph{2.Step:} Proof of the case $B=R[x_1,\dots,x_d]$.
In this case it follows from \cite[\textsection2]{LZ} and the proof of \cite[Theorem~3.5]{LZ} that 
$$
\Omega^{*}_{W_n(R)[x_1,\dots,x_d]/W_n(R)}\xr{} \Omega^{*}_{W_n(B)/W_n(R)} \xr{\pi} W_n\Omega^*_{B/R}
$$
is a quasi-isomorphism, where the first morphism is induced by $x_i\mapsto [x_i]$. The same statement holds for $R/I$, hence
the assertion follows from the quasi-isomorphism 
$$
\Omega^{*}_{W_n(R)[x_1,\dots,x_d]/W_n(R)} \otimes^{\mathbb{L}}_{W_n(R)} W_n(R/I)\xr{} \Omega^{*}_{W_n(R/I)[x_1,\dots,x_d]/W_n(R/I)}.
$$ 
\end{proof}

\begin{corollary}\label{corollary-comparison-thm-RGamma}
Let $R$ be a flat and finitely generated $\Z$-algebra, and let $\mf{m}\subset R$ be a maximal ideal; set $p={\rm char}(R/\mf{m})$. 
Let $X$ be a smooth and proper $R$-scheme, let $n,j$ be positive integers. 
There is a natural quasi-isomorphism of complexes of $W_n(R)$-modules:
$$
R\Gamma(W_n\Omega^*_{X/R}) \otimes^{\mathbb{L}}_{W_n(R)} W_n(R/\mf{m}^j) \xr{} R\Gamma(W_n\Omega^*_{X\otimes_{R} R/\mf{m}^j / (R/\mf{m}^j)}).
$$
\begin{proof}
The claim follows from Theorem \ref{thm-comparison-isom} by using \v{C}ech complexes (see \ref{subsubsection-can-be-computed-by-Cech-covering}).
\end{proof}
\end{corollary}

\subsection{Flatness}
\begin{thm}\label{thm-projective-blue}
  Let $R$ be a smooth $\Z$-algebra. Let $X$ be a smooth and proper $R$-scheme. Suppose that the de Rham cohomology $H^*_{dR}(X/R)$ of $X$
is a flat $R$-module. Then $H^*_{dRW}(X/\W_S(R))$ is a finitely generated 
projective $\W_S(R)$-module for all finite truncation sets $S$. Moreover, for an inclusion of finite truncation sets $T\subset S$, the induced map
\begin{equation}\label{equation-from-S-to-T}
H^*_{dRW}(X/\W_S(R))\otimes_{\W_S(R)}\W_T(R)\xr{\cong} H^*_{dRW}(X/\W_T(R))  
\end{equation}
is an isomorphism.
\end{thm}

Since $\W_S(R)$ is a noetherian ring and we know that $H^*_{dRW}(X/\W_S(R))$
is a finitely generated $\W_S(R)$-module (Proposition \ref{proposition-finiteness}), it remains to show that it is flat. 
This is a local property and can be checked after localization at maximal ideals of $\W_S(R)$. 
Our proof relies on Theorem \ref{thm-comparison-isom} or, more precisely, 
Corollary \ref{corollary-comparison-thm-RGamma}. 
 
\begin{lemma}\label{lemma-faithfully-flat-blue} 
  Let $R$ be a finitely generated $\Z$-algebra. Let $\mathfrak{m}$ be a maximal ideal of $R$, let $n$ be a positive
integer, and set $p={\rm char}(R/\mf{m})$. 
Then 
$W_n(R_{\mathfrak{m}})\xr{} W_n(\varprojlim_{i} R/\mathfrak{m}^i)$ is faithfully flat.
\begin{proof}
By Lemma \ref{lemma-WSRp-local-more-general}, both rings are local. Thus we only need to prove flatness. 

We note that $W_n(R)$ is a noetherian ring, because $R$ is a finitely generated $\Z$-algebra. Thus
$W_n(R_{\mathfrak{m}})$, being a localization of $W_n(R)$, is a noetherian ring.  

Obviously, we have the equalities
$$
W_n(\varprojlim_i R/\mf{m}^i)=\varprojlim_i W_n(R/\mf{m}^i) = \varprojlim_i W_n(R_{\mf{m}})/W_n(\mf{m}^iR_{\mf{m}}).
$$
Moreover, it is easy to check that $(W_n(\mf{m}^iR_{\mf{m}}))_i$ and $(W_n(\mf{m}R_{\mf{m}})^i)_i$ induce the same topology on $W_n(R_{\mf{m}})$.
Therefore 
\begin{equation}\label{equation-rewrite-as-adic-completion}
\varprojlim_i W_n(R_{\mf{m}})/W_n(\mf{m}R_{\mf{m}})^i \xr{\cong} \varprojlim_i W_n(R_{\mf{m}})/W_n(\mf{m}^iR_{\mf{m}}),  
\end{equation}
which implies flatness.
\end{proof}
\end{lemma}

\begin{lemma}\label{lemma-going-to-the-projective-limit-blue}
 Let $R$ be a finitely generated $\Z$-algebra. Let $\mathfrak{m}$ be a maximal ideal of $R$,  let $n$ be a positive
integer, and set $p={\rm char}(R/\mf{m})$.
Let $C$ be a bounded complex of $W_n(R_{\mathfrak{m}})$-modules such that
$H^i(C)$ is a finitely generated $W_n(R_{\mathfrak{m}})$-module for all $i$. 
Then, for all $i$, 
$$
H^i(C)\otimes_{W_n(R_{\mathfrak{m}})}  W_n(\varprojlim_{j} R/\mathfrak{m}^j) \cong \varprojlim_j
H^i\left(C\otimes^{\mathbb{L}}_{W_n(R_{\mathfrak{m}})} W_n(R/\mathfrak{m}^j)\right).
$$
\begin{proof}
Set $\hat{R}:=\varprojlim_{j} R/\mathfrak{m}^j$. 
The map is induced by $C\xr{} C\otimes^{\mathbb{L}}_{W_n(R_{\mathfrak{m}})} W_n(R/\mathfrak{m}^j)$
and the $ W_n(\hat{R})$-module structure on the right hand side.

As a first step we will prove that $H^i\left(C\otimes^{\mathbb{L}}_{W_n(R_{\mathfrak{m}})} W_n(R/\mathfrak{m}^j)\right)$ is a finite group. 
Clearly, we may assume that $C=C_0$ is concentrated in degree $0$. Since $C_0$
is finitely generated we conclude that ${\rm Tor}_i^{W_n(R_{\mathfrak{m}})}(C_0,W_n(R/\mathfrak{m}^j))$ is a finitely
generated $W_n(R/\mathfrak{m}^j)$-module for all $i$. The ring $W_n(R/\mathfrak{m}^j)$ contains only finitely
many elements, hence $$H^{-i}(C\otimes^{\mathbb{L}}_{W_n(R_{\mathfrak{m}})} W_n(R/\mathfrak{m}^j))={\rm Tor}_i(C_0,W_n(R/\mathfrak{m}^j))$$ is finite. 

By using Lemma \ref{lemma-faithfully-flat-blue} and the first step (all $R^1\varprojlim$ vanish) we can reduce the assertion to the case of a complex $C=C_0$ that
is concentrated in degree zero (hence $C_0$ is finitely generated). In this case we need to show:
\begin{itemize}
\item [(a)] $C_0\otimes_{W_n(R_{\mathfrak{m}})}W_n(\hat{R})\xr{=}\varprojlim_{j}(C_0\otimes_{W_n(R_{\mathfrak{m}})}W_n(R/\mathfrak{m}^j))$,
\item [(b)] $\varprojlim_{j}{\rm Tor}_i(C_0,W_n(R/\mathfrak{m}^j))=0$ for all $i>0$. 
\end{itemize}
Claim (a) follows from (\ref{equation-rewrite-as-adic-completion}).  Claim (b) follows from (a) and the flatness of  $W_n(R_{\mathfrak{m}}) \xr{} W_n(\hat{R})$. 
\end{proof}
\end{lemma}

\begin{proposition}\label{proposition-de-Rham-Witt-cohomology-limit-blue}
Assumptions as in Corollary \ref{corollary-comparison-thm-RGamma}. 
Set $X_j:=X\otimes_{R}R/\mathfrak{m}^j$, $R_j:=R/\mathfrak{m}^j$, $\hat{R}=\varprojlim_j R_j$.
\begin{itemize}
\item [(i)] For all $i$ and $n$, we have a functorial isomorphism 
\begin{equation}\label{equation-de-Rham-Witt-cohomology-limit-blue}
H^i_{dRW}(X/W_n(R))\otimes_{W_n(R)} W_n(\hat{R})\xr{\cong} \varprojlim_j H^i(X_j,W_n\Omega^*_{X_j/R_j}).   
\end{equation}
\item [(ii)] Suppose furthermore that the following conditions are satisfied:
  \begin{enumerate}
  \item There exists a lifting $\phi:\hat{R}\xr{} \hat{R}$ of the absolute Frobenius on $R/\mf{m}$; let $\rho:\hat{R}\xr{} W_n(\hat{R})$ be the induced ring homomorphism. By abuse of notation we will denote the restriction of $\rho$ to $R$ by $\rho$ again.
  \item The de Rham cohomology $H^*_{dR}(X/R)$ is a locally free $R$-module.
  \end{enumerate}  
Then there is an isomorphism 
$$
H^i(X_j,W_n\Omega^*_{X_j/R_j})  \cong  H^i_{dR}(X/R)\otimes_{R,\rho} W_n(R_j) 
$$ 
which is natural in the following sense. For all $l>j$ we have a commutative diagram
$$
\xymatrix{
H^i(X_l,W_n\Omega^*_{X_l/R_l})\ar[r]^-{\cong}\ar[d]
&  
H^i_{dR}(X/R)\otimes_{R,\rho} W_n(R_l)\ar[d]^{id\otimes W_n(R_l\xr{} R_j)}
\\
H^i(X_j,W_n\Omega^*_{X_j/R_j}) \ar[r]^-{\cong} 
&  
H^i_{dR}(X/R)\otimes_{R,\rho} W_n(R_j).
}
$$
For a morphism of $R$-schemes $f:X\xr{} Y$, where $Y/R$ satisfies the same assumptions as $X$, the following diagram is commutative:
$$
\xymatrix{
H^i(Y_j,W_n\Omega^*_{Y_j/R_j}) \ar[r]^-{\cong} \ar[d]^{f^*}
&  
H^i_{dR}(Y/R)\otimes_{R,\rho} W_n(R_j) \ar[d]^{f^*\otimes id}
\\
H^i(X_j,W_n\Omega^*_{X_j/R_j}) \ar[r]^-{\cong}  
&  
H^i_{dR}(X/R)\otimes_{R,\rho} W_n(R_j).
}
$$
\end{itemize}
\begin{proof}
For (i). Set $C=R\Gamma(W_n\Omega^*_{X/R})\otimes_{W_n(R)} W_n(R_{\mf{m}})$. 
In view of Proposition \ref{proposition-finiteness}, the assumptions for Lemma \ref{lemma-going-to-the-projective-limit-blue}
are satisfied. Applying the lemma and using Corollary \ref{corollary-comparison-thm-RGamma} implies the claim.

For (ii).   Consider the following cartesian squares
$$
\xymatrix
{
X_j \ar[r]\ar[d]
&
X_{n,j} \ar[r]\ar[d]
&
X\otimes_R \hat{R}\ar[d]
\\
\Spec(R_j) \ar[r]^-{{\rm gh}_1}
&
\Spec(W_n(R_j)) \ar[r]^-{\rho}
&
\Spec(\hat{R}),
}
$$
where $X_{n,j}$ is by definition the fibre product. Note that $\hat{R}\xr{\rho} W_n(\hat{R})\xr{{\rm gh}_1}\hat{R}$ is the identity, which implies
that the left hand square is cartesian.

By the comparison theorem \cite[Theorem~3.1]{LZ} we have a functorial isomorphism
$$
H^i(X_j,W_n\Omega^*_{X_j/R_j}) \cong H^i_{crys}(X_j/W_n(R_j)).
$$
By the comparison isomorphism
of crystalline cohomology with de Rham cohomology due to Berthelot-Ogus we get
\begin{align*}
  H^i_{crys}(X_j/W_n(R_j)) & \cong H^i_{dR}(X_{n,j}/W_n(R_j)) \\
                         & \cong H^i_{dR}(X/R)\otimes_{R,\rho} W_n(R_j).  
\end{align*}
For the last isomorphism we have used condition (2) on the de Rham cohomology of $X$. 
\end{proof}
\end{proposition}

\begin{proof}[Proof of Theorem \ref{thm-projective-blue}]
Without loss of generality we may assume that $R$ is integral. 
  It suffices to show the flatness of $H^i_{dRW}(X/\W_S(R))$ when considered as 
a $\W_S(R)$-module. This can be checked after localizing at maximal ideals. 
By using Lemma \ref{lemma-points-of-WSR} it suffices to prove that $H^i_{dRW}(X/\W_S(R))\otimes_{\W_S(R)} \W_S(R_{\mf{m}})$ is a flat 
$\W_S(R_{\mf{m}})$-module for every maximal ideal $\mf{m}\subset R$. Similarly, it is sufficient to prove (\ref{equation-from-S-to-T}) after
tensoring with $\W_T(R_{\mf{m}})$.

Let $\mf{m}\subset R$ be a maximal ideal, and set $p={\rm char}(R/\mf{m})$.
By using the decomposition of $\W_S\Omega^*_{X/R}\otimes \Z_{(p)}$ from Proposition \ref{proposition-epsilon-decomposion-deRham-Witt} together with \eqref{equation-epsilon-decomposition} we may assume that $S$ is $p$-typical, say $S=\{1,p,\dots,p^{n-1}\}$, and hence $T=\{1,p,\dots,p^{m-1}\}$.

Since $R$ is a smooth $\Z$-algebra, there is a lifting $\phi:\hat{R}\xr{}\hat{R}$ of the absolute 
Frobenius of $R/\mf{m}$, where $\hat{R}=\varprojlim_j R/\mf{m}^j$. Therefore Proposition \ref{proposition-de-Rham-Witt-cohomology-limit-blue} implies
\begin{multline*}
H^i_{dRW}(X/W_n(R))\otimes_{W_n(R)} W_n(\hat{R})\xr{\cong} \varprojlim_j H^i(X_j,W_n\Omega^*_{X_j/R_j}) \\
\xr{\cong} H^i_{dR}(X/R)\otimes_{R,\rho} W_n(\hat{R}),     
\end{multline*} 
and we can prove the flatness by using Lemma \ref{lemma-faithfully-flat-blue}.

Tensoring (\ref{equation-from-S-to-T}) with $W_{m}(\hat{R})$ (recall that $T=\{1,p,\dots,p^{m-1}\}$) and by using Proposition \ref{proposition-de-Rham-Witt-cohomology-limit-blue}(ii), we see that $\text{(\ref{equation-from-S-to-T})}\otimes W_{m}(\hat{R})$ is induced by the identity on the de Rham cohomology. Hence it is an 
isomorphism by Lemma \ref{lemma-faithfully-flat-blue}.
\end{proof}

\section{Poincar\'e duality} \label{section-values}

\subsection{A rigid $\otimes$-category}

\begin{definition} \label{definition-phi-N}
  Let $R$ be a $\Z$-torsion-free ring  and $Q$ a non-empty truncation set. We denote by $\mathcal{C}'_{Q,R}$
the category with objects being contravariant functors $S\mapsto M_S$ from finite truncation 
sets contained in $Q$ to sets, together with 
\begin{itemize} 
\item a $\W_S(R)$-module structure on $M_S$, for all truncation sets $S\subset Q$, such that
the maps $M_S\xr{} M_T$, for $T\subset S$, are morphisms of $\W_S(R)$-modules when $M_T$ is considered as
a $\W_S(R)$-module via the projection $\pi_T:\W_S(R)\xr{} \W_T(R)$,
\item for all positive integers $n$ and all truncation sets $S\subset Q$, maps
$$
\phi_n:M_S\xr{} M_{S/n},
$$
such that 
\begin{itemize}
\item $\phi_{n}\circ \phi_m=\phi_{nm}$ for all $n,m$,
\item $\phi_n$ is a morphism of $\W_S(R)$-modules when $M_{S/n}$ is considered
as a $\W_S(R)$ module via $F_n:\W_S(R)\xr{} \W_{S/n}(R)$,
\item for all truncation sets $T\subset S\subset Q$ the following diagram is commutative: 
$$
\xymatrix{
M_S\ar[r]^{\phi_n}\ar[d]
&
M_{S/n}\ar[d]
\\
M_T\ar[r]^{\phi_n}
&
M_{T/n}.
}
$$
\end{itemize}
\end{itemize}
The functor $S\mapsto M_S$ is required to satisfy the following properties.
\begin{itemize}
\item For all truncation sets $S\subset Q$, the $\W_S(R)$-module $M_S$ is finitely generated and projective. 
\item For all truncation sets $T\subset S\subset Q$:
$$
\W_T(R)\otimes_{\W_S(R)}M_S \xr{} M_T
$$
is an isomorphism.
\item There is a positive integer $a$ such that there exist morphisms 
  \begin{equation}\label{equation-beta_N}
\beta_n:M_{S/n}\xr{} M_{S},    
  \end{equation}
for all positive integers $n$ and all finite truncation sets $S\subset Q$, satisfying the following properties:
\begin{itemize}
\item $\beta_n$ is a morphism of $\W_S(R)$-modules when $M_{S/n}$ is considered
as a $\W_S(R)$ module via $F_n:\W_S(R)\xr{} \W_{S/n}(R)$,
\item $\beta_n(\lambda \cdot \phi_n(x))=n^{a-1}V_n(\lambda)\cdot x,$ for all $x\in M_S,\lambda\in \W_{S/n}(R)$,
\item $\phi_n\circ \beta_n=n^{a}.$
\end{itemize}
\end{itemize}

Morphisms between two objects in $\mathcal{C}'_{Q,R}$ are morphism of functors 
that are compatible with the $[S\mapsto \W_S(R)]$-module structure and 
commute with  $\phi_n$ for all positive integers $n$. We simply write $\mathcal{C}'_{R}$ for $\mathcal{C}'_{\mathbb{N}_{>0},R}$.
\end{definition}

\begin{remark}
Note that the $\beta_n$ are not part of the datum; we can always change
$\beta_n\mapsto n^{b}\beta_n$ for a non-negative integer $b$.  
\end{remark}

For an inclusion of truncation sets $Q\subset Q'$, we have an evident functor 
$$
\mathcal{C}'_{Q',R}\xr{}  \mathcal{C}'_{Q,R}.
$$

\begin{proposition}
Let $M\in {\rm ob}(\mathcal{C}'_{Q,R})$. Let $S\subset Q$ be a finite truncation set. Fix $a>0$ and $\beta_n$ as in \ref{equation-beta_N}. 
\begin{enumerate}
\item For all positive integers $n,m$ with $(n,m)=1$ we have 
$$\phi_n\circ \beta_m=\beta_m\circ \phi_n,$$
considered as morphisms $M_{S/m}\xr{} M_{S/n}$.
\item For all positive integers $n,m$ we have 
$$\beta_{n}\circ \beta_m=\beta_{nm},$$
considered as morphisms $M_{S/nm}\xr{} M_{S}$.
\item For all truncation sets $T\subset S$ the following diagram is commutative: 
$$
\xymatrix{
M_{S/n}\ar[r]^{\beta_n}\ar[d]
&
M_{S}\ar[d]
\\
M_{T/n}\ar[r]^{\beta_n}
&
M_{T}.
}
$$
\end{enumerate}
  \begin{proof}
    The ring $\W_S(R)$ is $\Z$-torsion-free, because  
it can be considered via the ghost map as a subring of $\prod_{s\in S}R$, and $R$ is $\Z$-torsion-free by assumption.  Since $M_S$ is a flat $\W_S(R)$-module, it is $\Z$-torsion-free. 

For (1). Since ${\rm image}(\phi_m)\supset m^aM_{S/m}$ it is sufficient to prove
$$
\phi_n\circ \beta_m\circ \phi_m=\beta_m\circ \phi_n\circ \phi_m.
$$
This follows from $\beta_m\circ \phi_m=V_m(1)m^{a-1}$ and $\phi_n\circ \phi_m=\phi_m\circ \phi_n$.

For (2). We may argue as in (1) by composing with $\circ \phi_{nm}$.
$$
\beta_{n}\circ \beta_m \circ \phi_{nm}(x)= 
\beta_{n}(V_{m}(1)m^{a-1}\phi_{n}(x))
=m^{a-1}n^{a-1}V_{nm}(1)x \\
=\beta_{nm}\circ \phi_{nm}(x).
$$
For (3). We may argue as in (1) by composing with $\circ \phi_{n}$. The computation 
is straightforward.
  \end{proof}
\end{proposition}

\begin{lemma}\label{lemma-commute-with-beta}
 Let $f:M\xr{} N$ be a morphism in $\mathcal{C}'_{Q,R}$,  
and choose a positive integer $a$ and $\beta_{M,n}$, $\beta_{N,n}$ as in (\ref{equation-beta_N}).
  Then $f_S\circ \beta_{M,n}=\beta_{N,n}\circ f_{S/n}$ for all $S,n$.
In particular, the choice of the $\beta_{n}$ in Definition \ref{definition-phi-N} depends only on the positive integer $a$.  
\begin{proof}
Again, we may use that $M_S$ is $\Z$-torsion-free. Now, 
  \begin{align*}
    n^a\beta_nf(x)&=\beta_n(f(n^ax))
                  =\beta_n(f(\phi_n\beta_n(x)))\\
                  &=\beta_n\phi_nf(\beta_n(x))
                  =n^{a-1}V_n(1)f(\beta_n(x))\\
                  &=f(n^{a-1}V_n(1)\beta_n(x))
                  =f(\beta_n\phi_n\beta_n(x))
                  =n^af(\beta_n(x)).
  \end{align*}
\end{proof}
\end{lemma}

\begin{proposition}[Tensor products] 
  For two objects $M,N$ in $\mathcal{C}'_{Q,R}$ we set 
$$
(M\otimes N)_S:=M_S\otimes_{\W_S(R)} N_S, \quad \phi_n:=\phi_{M,n}\otimes \phi_{N,n}.
$$
Then $M\otimes N$ defines an object in $\mathcal{C}'_{Q,R}$.
\begin{proof}
 This is a straightforward calculation.  We can take 
$\beta_{M\otimes N,n}=\beta_{M,n}\otimes \beta_{N,n}$.
\end{proof}
\end{proposition}
The tensor product equips $\mathcal{C}'_{Q,R}$ with the structure of a $\otimes$-category with identity object $\mathbf{1}$, where
$$
\mathbf{1}_S:=\W_S(R), \qquad \phi_{\mathbf{1},n}=F_n.
$$

\begin{definition}(Tate objects)
  Let $b$ be a non-negative integer. We define the object $\mathbf{1}(-b)$ in $\mathcal{C}'_{Q,R}$ by 
$$
\mathbf{1}(-b)_S:=\W_S(R), \qquad \phi_{\mathbf{1}(-b),n}=n^bF_n.
$$ 
\end{definition}

For an object  $M$ in $\mathcal{C}'_{Q,R}$, $M_S$ is $\Z$-torsion-free, hence we get an isomorphism:
$$
\Hom_{\mathcal{C}'_{Q,R}}(M,N)\xr{\cong} \Hom_{\mathcal{C}'_{Q,R}}(M\otimes \mathbf{1}(-b),N\otimes \mathbf{1}(-b))
$$

\begin{definition} 
  We denote by $\mathcal{C}_{Q,R}$ the category with objects $M(b)$, where $M$
is an object in  $\mathcal{C}'_{Q,R}$ and $b\in \Z$. As morphisms we set 
$$
\Hom_{\mathcal{C}_{Q,R}}(M(b_1),N(b_2))=\Hom_{\mathcal{C}'_{Q,R}}(M\otimes \mathbf{1}(b_1-c),N\otimes \mathbf{1}(b_2-c)),
$$
where $c\in \Z$ is such that $b_1-c,b_2-c\leq 0$.
\end{definition}

For two truncation sets $Q\subset Q'$, we have an obvious functor
$$
\mathcal{C}_{Q',R} \xr{} \mathcal{C}_{Q,R}.
$$

The category $\mathcal{C}_{Q,R}$ is  additive and via $M\mapsto M(0)$ the 
category $\mathcal{C}'_{Q,R}$ is a full subcategory of $\mathcal{C}_{Q,R}$. For
$M\in \mathcal{C}'_{Q,R}$, we have $M(-b)=M\otimes \mathbf{1}(-b)$ if $b$ is non-negative. 
For an integer $b$, the functor 
\begin{align*}
  \mathcal{C}_{Q,R}\xr{} \mathcal{C}_{Q,R}, \quad M(n)\mapsto M(n+b)
\end{align*}
is an equivalence and has $M(n)\mapsto M(n-b)$ as inverse functor. 

For $M(b_1),N(b_2)$ in $\mathcal{C}_{Q,R}$ we set 
$$
M(b_1)\otimes N(b_2):=(M\otimes N)(b_1+b_2).
$$
The tensor product equips $\mathcal{C}_{Q,R}$ with the structure of a $\otimes$-category with identity object $\mathbf{1}$.

\subsubsection{Internal Hom}
The reason for introducing the new category $\mathcal{C}_{Q,R}$ is the internal Hom
construction.

Let $M,N$ be two objects in $\mathcal{C}'_{Q,R}$, fix positive integers $a_M,a_N$
and $\beta_{n,M},\beta_{n,N}$ as in (\ref{equation-beta_N}). In a first step we 
are going to define
an object $\iHom'(M,N)$ in  $\mathcal{C}'_{Q,R}$ that depends on the choice of $a_M$. We set
$$
\iHom'(M,N)_S:=\Hom_{\W_S(R)}(M_S,N_S).
$$
We note that  
$$
\Hom_{\W_S(R)}(M_S,N_S)\otimes_{\W_S(R)}\W_T(R)\xr{\cong} \Hom_{\W_T(R)}(M_T,N_T),
$$
since $M_S$ is finitely generated and projective. 
We define 
\begin{align*}
\phi_n:\Hom_{\W_S(R)}(M_S,N_S)&\xr{} \Hom_{\W_{S/n}(R)}(M_{S/n},N_{S/n}) \\ 
\phi_n(f)&:=\phi_n\circ f\circ \beta_n.
\end{align*}
This definition depends on $a_M$. It is easy to check that $\iHom'(M,N)$
is an object in $\mathcal{C}'_{Q,R}$ (take $\beta_n(f):=\beta_n\circ f\circ \phi_n$
and $a=a_M+a_N$). We set 
\begin{equation}\label{equation-definition-iHom}
  \iHom(M,N):=\iHom'(M,N)(a_M) 
\end{equation}
as an object in $\mathcal{C}_{Q,R}$. In view of Lemma \ref{lemma-commute-with-beta}
this definition is independent of any choices. For two objects $M(b_1),N(b_2)$
in $\mathcal{C}_{Q,R}$ we set
$$
\iHom(M(b_1),N(b_2)):= \iHom(M,N)(b_2-b_1).
$$

\subsubsection{}
For three objects $M,N,P$ in $\mathcal{C}_{Q,R}$ we have an obvious natural 
isomorphism 
$$
\iHom(M\otimes N,P) = \iHom(M,\iHom(N,P)).
$$
\begin{proposition}
For objects $M,N$ in   $\mathcal{C}_{Q,R}$  we have a natural isomorphism
$$
\Hom(\mathbf{1},\iHom(M,N))\xr{} \Hom(M,N).
$$ 
\begin{proof}
We may assume that $M,N\in \mathcal{C}'_{Q,R}$.
  Fix $a_M$ and $\beta_{M,n}$ as in (\ref{equation-beta_N}). We need to show 
that 
$$
\Hom(\mathbf{1}(-a_M),\iHom'(M,N))=\Hom(M,N),
$$
and know that 
\begin{multline*}
\Hom(\mathbf{1}(-a_M),\iHom'(M,N))=\{[S\mapsto f_S]\mid f_S\otimes_{\W_S(R)}\W_T(R)=f_T\;\text{for $T\subset S\subset Q$}, \\ 
\phi_{N,n}\circ f_S\circ \beta_{M,n} = n^{a_M}f_{S/n} \quad \text{for all $n, S\subset Q$.}\}   
\end{multline*}
Since $\phi_{M,n}(M_S)\supset n^{a_M}M_{S/n}$ we have 
\begin{align*}
  \phi_{N,n}\circ f_S\circ \beta_{M,n} = n^{a_M}f_{S/n}  
&\Leftrightarrow \phi_{N,n}\circ f_S\circ \beta_{M,n}\circ \phi_{M,n} = n^{a_M}f_{S/n}\circ \phi_{M,n} \\
&\Leftrightarrow \phi_{N,n}\circ f_S\circ n^{a_{M}-1} V_n(1) = n^{a_M}f_{S/n}\circ \phi_{M,n}\\
&\Leftrightarrow n^{a_{M}}\phi_{N,n}\circ f_S = n^{a_M}f_{S/n}\circ \phi_{M,n}\\
&\Leftrightarrow \phi_{N,n}\circ f_S = f_{S/n}\circ \phi_{M,n}.
\end{align*}
\end{proof}
\end{proposition}

For $M\in \mathcal{C}_{Q,R}$ we define the dual by 
$$
M^{\vee}:=\iHom(M,\mathbf{1}).
$$
It equips $\mathcal{C}_{Q,R}$ with the structure of rigid $\otimes$-category. We have
$$
M^{\vee}\otimes N = \iHom(M,N).
$$

\subsubsection{Functoriality}

\begin{proposition}\label{proposition-functoriality}
 Let $R\xr{} A$ be a ring homomorphism between $\Z$-torsion-free rings. 
The assignment
\begin{align*}
&[S\mapsto M_S]\mapsto [S\mapsto M_S\otimes_{\W_S(R)}\W_S(A)], \quad [n\mapsto \phi_n]\mapsto [n\mapsto \phi_n\otimes F_n],\\
&[S\mapsto f_S]\mapsto [S\mapsto f_S\otimes id_{\W_S(A)}]  
\end{align*}
defines a functor
$$
\mathcal{C}'_{Q,R}\xr{}  \mathcal{C}'_{Q,A}.
$$
The functor can be extended in the obvious way to a functor
$
\mathcal{C}_{Q,R}\xr{}  \mathcal{C}_{Q,A}.
$ 
\begin{proof}
  Straightforward.
\end{proof}
\end{proposition}

\subsubsection{}
Our motivation for introducing $\mc{C}_{Q,R}$ comes from geometry.

\begin{proposition}\label{proposition-big-de-Rham-Witt-phi-module}
 Assumptions as in Theorem \ref{thm-projective-blue}. Let $Q$ be a non-empty truncation set. For all $i\geq 0$ the assignment 
$$
S\mapsto H^i_{dRW}(X/\W_S(R)), \quad n\mapsto \phi_n,
$$
defines an object in $\mc{C}'_{Q,R}$.
\begin{proof}
  Theorem  \ref{thm-projective-blue} implies that these modules are projective
and finitely generated. For the construction of $\phi_n$ and $\beta_n$ see Section \ref{section-phi-beta-de-Rham-Witt}. 
\end{proof}
\end{proposition}

\begin{definition}\label{definition-bd-R}
  Let $X\xr{} \Spec(R)$ be a morphism such that the assumptions of Theorem \ref{thm-projective-blue} are satisfied. For all $i$, we denote by $H^i_{dRW}(X/\W(R))$ the object in $\mc{C}_{R}$
that is given by $S\mapsto H^i_{dRW}(X/\W_S(R))$ (Proposition \ref{proposition-big-de-Rham-Witt-phi-module}).
We call $H^*_{dRW}(X/\W(R))$ the \emph{de Rham-Witt cohomology} of $X$.
\end{definition} 

\subsubsection{}
Let $X,Y$ be smooth proper schemes over $R$ such that the assumptions of Theorem 
\ref{thm-projective-blue} are satisfied for $X$ and $Y$. 
The multiplication  
$$
R\Gamma(\W_S\Omega^*_{X/R})\times R\Gamma(\W_S\Omega^*_{Y/R})\xr{} R\Gamma(\W_S\Omega^*_{X\times_R Y/R})
$$ 
induces a morphism in $\mc{C}_R$:
\begin{equation}\label{equation-product-varities-C-R}
H^i_{dRW}(X/\W(R))\otimes H^{j}_{dRW}(Y/\W(R))\xr{}  H^{i+j}_{dRW}(X\times_R Y/\W(R)).
\end{equation}

\subsection{The tangent space functor} 

We have a functor of rigid $\otimes$-categories
\begin{align*}
&T:\mathcal{C}_{Q,R} \xr{} \text{(finitely generated and projective $R$-modules)} \\ 
   &T(M(n)) := M_{\{1\}}.
\end{align*}

\begin{proposition}\label{proposition-T-conservative}
 The functor $T$ is conservative, i.e.~if $T(f)$ is an isomorphism then $f$ is an isomorphism.  
  \begin{proof}
    It is sufficient to consider a morphism $f:M\xr{} N$ in $\mathcal{C}'_{Q,R}$.
We need to show that $f_S:M_S\xr{} N_{S}$ is an isomorphism provided that $f_{\{1\}}$ is an isomorphism. We may choose a positive integer $a$ and $\beta_{M,n}$, $\beta_{N,n}$ as in (\ref{equation-beta_N}). By Lemma \ref{lemma-commute-with-beta} the 
morphism $f$ commutes with $\beta_n$.

Let $n:=\max\{s\mid s\in S\}$; by induction we know that $f_T$ is an isomorphism for
$T=S\backslash \{n\}$. 
Set $I=\ker(\W_S(R)\xr{} \W_T(R))$, we know that $I=\{V_n(\lambda)\mid \lambda\in R\}$.
It suffices to show that 
\begin{equation}\label{equation-fS-restricted}
  IM_S\xr{f_S} IN_S 
\end{equation}
is an isomorphism. If $f_S(V_n(\lambda)x)=0$ then $n^{a-1}V_n(\lambda)f_S(x)=0$ and therefore
$\beta_n(\lambda\cdot \phi_nf_S(x))=\beta_n(\lambda\cdot f_{\{1\}}(\phi_n(x)))$ vanishes. Since $\beta_n$ is injective, we conclude 
$\lambda \cdot \phi_n(x)=0$, hence $$0=\beta_n(\lambda\cdot\phi_n(x))=n^{a-1}V_n(\lambda)x,$$ which implies $V_n(\lambda)x=0$.

For the surjectivity of (\ref{equation-fS-restricted}) we note that, by induction, for every 
$y\in N_S$ there is $x\in M_S$ with $f_S(x)-y\in IN_S$. Therefore it suffices
to show that $I^{a}N_S$ is contained in the image of $f_S$. Now,  
$$
V_n(\lambda_1)\cdots V_n(\lambda_a)=n^{a-1}V_n(\lambda_1\cdots\lambda_a).
$$
Thus 
$$
V_n(\lambda_1)\cdots V_n(\lambda_a)y=f_S(\beta_nf^{-1}_{\{1\}}(\lambda_1\cdots\lambda_a\cdot\phi_n(y))).
$$  
  \end{proof}
\end{proposition}

\begin{corollary}\label{corollary-kuenneth-over-R}
Let $X,Y$ be smooth proper schemes over $R$ such that the assumptions of Theorem \ref{thm-projective-blue} are satisfied for $X$ and $Y$. 
If 
$$
\bigoplus_{i+j=n} H^i_{dR}(Y/R)\otimes_R H^j_{dR}(X/R)\xr{} H^n_{dR}(X\times_R Y/R)
$$
is an isomorphism then 
$$
\bigoplus_{i+j=n} H^i_{dRW}(X/\W(R))\otimes H^j_{dRW}(Y/\W(R))\xr{} H^n_{dRW}(X\times_R Y/\W(R))
$$
(see \eqref{equation-product-varities-C-R}) is an isomorphism in $\mc{C}_{R}$.
\begin{proof}
  This is an application of Proposition \ref{proposition-T-conservative}, because
$$
T(H^i_{dRW}(-/\W(R)))=H^i_{dR}(-/R).
$$
\end{proof}
\end{corollary}

\begin{proposition}\label{proposition-T-faithful}
  The functor $T$ is faithful.  
  \begin{proof}
   It is sufficient to consider a morphism $f:M\xr{} N$ in $\mathcal{C}'_{Q,R}$.
We need to show that $f_S:M_S\xr{} N_{S}$ vanishes provided that $f_{\{1\}}$ is zero. We may choose a positive integer $a$ and $\beta_{M,n}$, $\beta_{N,n}$ as in (\ref{equation-beta_N}). By Lemma \ref{lemma-commute-with-beta} the 
morphism $f$ commutes with $\beta_n$.

Let $n:=\max(S)$; by induction we know that $f_T=0$  for
$T=S\backslash \{n\}$, so that for all $x\in M_S$ the image $f_S(x)$ is of the form 
$f_S(x)=V_n(\lambda)y$. Since
$$
0=f_{\{1\}}\circ \phi_n(x)=\phi_n\circ f_S(x)=n\cdot \lambda \cdot \phi_n(y),
$$
we conclude $\lambda\cdot \phi_n(y)=0$ and $n^{a-1}V_n(\lambda)y=0$, hence $f_S(x)=0$.
\end{proof}
\end{proposition}
 
\subsubsection{}
 The following proposition shows that an object in $\mathcal{C}'_{Q,R}$, where $R$ is a $\Z_{(p)}$-algebra, is determined by the $p$-typical part, that is, on its values for truncation sets consisting of $p$-powers. Recall the notation $S_p$ from Notation \ref{notation-S-p}.

\begin{proposition}\label{proposition-phiN-modules-over-Zp}
Let $R$ a $\Z$-torsion-free ring. Let $Q$ be a truncation set. Suppose $p$ is a prime such that $\ell^{-1}\in R$ for all primes $\ell\in Q\backslash \{p\}$.
  Let $M,N$ be $\mathcal{C}'_{Q,R}$-modules. 
 \begin{enumerate}
 \item Via the equivalence of \eqref{equation-epsilon-decomposition}: 
$$M_S\mapsto  \bigoplus_{n\in S, (n,p)=1} M_{(S/n)_p}.$$  
 \item If $f:M\xr{} N$ is a morphism in $\mathcal{C}'_{Q,R}$ then $f_S\mapsto  \bigoplus_{n\in S, (n,p)=1}f_{(S/n)_p}$
via the equivalence  \eqref{equation-epsilon-decomposition}.
\item The restriction functor 
$$
\mathcal{C}'_{Q,R}\xr{} \mathcal{C}'_{Q_p,R}
$$
is an equivalence of categories.
 \end{enumerate}
 \begin{proof}
For (1). First one proves that the projection 
$\epsilon_1M_S\xr{} M_{S_p}$ is an isomorphism (see Notation \ref{notation-S-p} for $S_p$). The second step is the isomorphism 
$$
\phi_n:\epsilon_nM_S\xr{} \epsilon_1M_{S/n}, 
$$
with $\frac{\epsilon_n}{n^a}\beta_n$ as inverse.

Statement (2) is obvious, and (3) follows from (1) and (2).
 \end{proof}
\end{proposition}

\begin{proposition}\label{proposition-R-over-Q-phi-N-modules}
Let $R$ a $\Z$-torsion-free ring. Let $Q$ be a truncation set. Suppose $p^{-1}\in R$ for all primes $p\in Q$.
Then  
$$
T:\mathcal{C}'_{Q,R} \xr{} \text{(finitely generated and projective $R$-modules)}
$$
defines an equivalence of categories. 
\begin{proof}
   Straightforward.
 \end{proof}
\end{proposition}

\subsubsection{}
Let $P$ be a set of primes (maybe infinite). We set $\Z_{P}:=\Z[p^{-1}\mid p\in P]$. Let $A$ be a commutative ring.
We denote by ${\rm Mod}_{A}$ the category of $A$-modules. We define the category ${\rm Mod}_{A,P}$ to be the category
with objects 
$$
((M_p)_{p\in P},(\alpha_{p,\ell})_{p,\ell\in P}),
$$
where $M_p$ is an $A\otimes_{\Z} \Z_{P\backslash \{p\}}$-module, and $\alpha_{p,\ell}:M_{\ell}\otimes_{\Z_{P\backslash \{\ell\}}}\Z_P\xr{} M_{p}\otimes_{\Z_{P\backslash \{p\}}}\Z_P$
is an isomorphism of $A\otimes \Z_P$-modules such that 
$$
\alpha_{p_1,p_1}=id, \quad
\alpha_{p_1,p_2}\circ \alpha_{p_2,p_3}=\alpha_{p_1,p_3}  \quad 
\text{for all $p_1,p_2,p_3\in P$.}
$$
The morphisms of ${\rm Mod}_{A,P}$ are defined in the evident way.

If $P$ is finite and non-empty, then the evident functor
$$
R_P:{\rm Mod}_{A}\xr{} {\rm Mod}_{A,P}
$$
is an equivalence of categories, because we can glue quasi-coherent sheaves. If $P$ is infinite then this may fail to be an equivalence, but we still 
have the following properties, whose proof is left to the reader.
\begin{lemma}\label{lemma-gluing-with-infinite-primes}
  Suppose $P\neq \emptyset$.
  \begin{itemize}
  \item [(i)] $R_P$ is faithful.
  \item [(ii)] For every $N\in {\rm Mod}_A$ such that $N\xr{} N\otimes_{\Z}\Z_P$ is injective, and every $M\in {\rm Mod}_A$ the following map
is an isomorphism:
$$
\Hom_{{\rm Mod}_A}(M,N)\xr{\cong} \Hom_{{\rm Mod}_{A,P}}(R_P(M),R_P(N)).
$$
\item [(iii)] Suppose that $A\xr{} A\otimes \Z_P$ is injective. Let $\tilde{M}=((\tilde{M}_p),(\alpha_{p,\ell}))\in {\rm Mod}_{A,P}$
be such that $\tilde{M}_p$ is a finitely generated and projective $A\otimes_{\Z} \Z_{P\backslash \{p\}}$-module for all $p\in P$. Then there exists a finitely generated and
projective $M\in {\rm Mod}_A$ such that
$R_P(M)\cong \tilde{M}$.  
  \end{itemize}
\end{lemma}

For a positive integer $a$, we denote by $\mc{C}'_{Q,R,a}$ the full subcategory of $\mc{C}'_{Q,R}$ consisting of objects such that there exist $\{\beta_n\}_{n}$ as in \eqref{equation-beta_N} for $a$. 
\begin{definition}
Let $Q$ be a non-empty truncation set, and let $P$ be the set of primes of $Q$. 
We denote by $\mc{LC}'_{Q,R,a}$ the category with objects $$((M_p)_{p\in P},(\alpha_{p,\ell})_{p,\ell \in P}),$$ where 
\begin{itemize}
\item $M_p\in {\rm ob}(\mc{C}'_{Q_p,R\otimes \Z_{P\backslash \{p\}},a})$ for all $p\in P$,
\item $\alpha_{p,\ell}:T(M_{\ell})\otimes_{\Z_{P\backslash \{\ell\}}} \Z_P \xr{\cong} T(M_{p})\otimes_{\Z_{P\backslash \{p\}}} \Z_P$ is an isomorphism such that
$$
\alpha_{p_1,p_1}=id, \quad
\alpha_{p_1,p_2}\circ \alpha_{p_2,p_3}=\alpha_{p_1,p_3}  \quad 
\text{for all $p_1,p_2,p_3\in P$.}
$$  
\end{itemize}  
The morphisms are defined in the evident way.
\end{definition}

Broadly speaking the next proposition shows that the category $\mc{C}'_{Q,R,a}$ 
is glued from the local components via the functor $T$.   
\begin{proposition}
 Let $R$ be a $\Z$-torsion-free ring, and let $a$ be a positive integer. For every non-empty truncation set $Q$ the evident functor
$$
\mc{C}'_{Q,R,a}\xr{} \mc{LC}'_{Q,R,a}
$$
is an equivalence of categories.
\begin{proof}
  The claim follows easily from Proposition \ref{proposition-phiN-modules-over-Zp}, Proposition \ref{proposition-R-over-Q-phi-N-modules}, and Lemma \ref{lemma-gluing-with-infinite-primes}.
\end{proof}
\end{proposition}

\subsection{Proof of Poincar\'e duality}

\subsubsection{}
Let $f:X\xr{} \Spec(R)$ be a smooth, projective morphism of relative dimension $d$ between noetherian schemes such that $H^*_{dR}(X/R)$ is a flat $R$-module. Suppose furthermore that $\Spec(R)$ is integral and the field of fractions of $R$ has characteristic zero.

We know that $H^0(X,\OO_X)$ is a finite \'etale $R$-algebra and $$H^0_{dR}(X/R)=H^0(X,\OO_X).$$ Since $H^*_{dR}(X/R)$ is 
flat, we have 
$$
H^i_{dR}(X/R)\otimes_R k(y)\xr{\cong} H^i_{dR}(X_y/k(y)),
$$
for every point $y\in \Spec(R)$, and $X_y$ being the fibre of $y$. In particular, we obtain
\begin{equation}\label{equation-H0-on-fibres}
H^0(X,\OO_X)\otimes_R k(y)\xr{\cong} H^0(X_y,\OO_{X_y}).  
\end{equation}
By Grothendieck-Serre duality we see that $y\mapsto \dim_{k(y)} H^d(X_y,\omega_{X_y})$ is a constant function, thus $H^d(X,\omega_{X/R})$ is a finitely generated 
projective $R$-module and we have 
$$
H^d(X,\omega_{X/R})\otimes_R k(y)\xr{\cong} H^d(X_y,\omega_{X_y})
$$
for every point $y\in \Spec(R)$. Since the Hodge to de Rham spectral sequence degenerates at the generic point of $\Spec(R)$, we conclude:
$$
H^d(X,\omega_{X/R})\xr{\cong} H^{2d}_{dR}(X/R).
$$
Recall that we have a trace map
$$
{\rm Tr}:H^d(X,\omega_{X/R})\xr{} R;
$$
we will also denote by ${\rm Tr}$ the induced map $H^{2d}_{dR}(X/R)\xr{} R$.
The duality pairing 
$$
H^0(X,\OO_X) \times H^d(X,\omega_{X/R})\xr{} R
$$
induces a duality pairing 
$$
H^0_{dR}(X/R)\times H^{2d}_{dR}(X/R) \xr{} R.
$$
Note that the fibres of $f$ are connected if $H^0(X,\OO_X)=R$. Moreover, the equality $H^0(X,\OO_X)=R$ implies that the fibres
are geometrically connected by using (\ref{equation-H0-on-fibres}). 

Suppose now that $H^0(X,\OO_X)=R$, and set $c_X:={\rm Tr}^{-1}(1)\in H^d(X,\omega_{X/R})=H^{2d}_{dR}(X/R)$. For a generically finite $R$-morphism 
$g:X\xr{} Y$, where $Y$ satisfies the same assumptions as $X$ (in particular, $Y/R$ is of relative dimension $d$), 
we have a pull-back map
$$
g^*:H^{2d}_{dR}(Y/R)=H^d(Y,\omega_{Y/R})\xr{} H^d(X,\omega_{X/R}) = H^{2d}_{dR}(X/R)
$$
which is dual to the trace map 
$$
g_*:H^0(X,\OO_X)\xr{} H^0(Y,\OO_Y), \quad g_*(1)=\deg(g),
$$ 
thus 
$
g^*(c_Y)=\deg(g)\cdot c_X.
$ 

\begin{proposition}\label{proposition-trace-map}
 Let $R$ be a smooth $\Z$-algebra.
  Let $X$ be a smooth projective scheme over $R$ such that $H^*_{dR}(X/R)$ is a projective $R$-module. 
Suppose that $X$ is connected of relative dimension $d$. There is an isomorphism 
$$
H^{2d}_{dRW}(X/\W(R))\cong H^0_{dRW}(X/\W(R))\otimes \mathbf{1}(-d)
$$
 and a natural morphism  in $\mathcal{C}_R$:
$$
 H^{2d}_{dRW}(X/\W(R)) \xr{} \mathbf{1}(-d)
$$
 \begin{proof} 
Certainly, we may suppose that $\Spec(R)$ is integral.

  \emph{1.Step:} Reduction to $X/R$ has geometrically connected fibres. 

Set $L=H^0(X,\OO_X)$, $L$ is a finite \'etale $R$-algebra. It suffices to show the existence
of an isomorphism 
\begin{equation}\label{equation-reduction-to-geom-int-case}
 H^{2d}_{dRW}(X/\W(L)) \xr{\tau} \mathbf{1}(-d)  
\end{equation}
in $\mathcal{C}_L$ such that $\tau_{\{1\}}$ is the trace map. In view of $$H^{2d}_{dRW}(X/\W(L))=H^{2d}_{dRW}(X/\W(R)),$$ (\ref{equation-reduction-to-geom-int-case}) yields in $\mathcal{C}_R$:
\begin{equation}\label{equation-induced-trace-map-on-R}
H^{2d}_{dRW}(X/\W(R))\xr{\cong} H^0_{dRW}(X/\W(R))\otimes \mathbf{1}(-d) \xr{tr\otimes id} \mathbf{1}(-d),  
\end{equation}
with $tr:H^0_{dRW}(X/\W(R)) \xr{} \mathbf{1}$ being defined by the usual trace map 
$$
H^0_{dRW}(X/\W_S(R))=\W_S(L) \xr{} \W_S(R).
$$ 
The morphism (\ref{equation-induced-trace-map-on-R}) is functorial because it induces the usual trace map after evaluation at $\{1\}$.
Therefore we may assume $R=L$ in the following. 

  \emph{2.Step:} Proposition \ref{proposition-R-over-Q-phi-N-modules} implies the existence of a unique isomorphism 
$$
e:\mathbf{1}(-d)\otimes \Q \xr{\cong} H^{2d}_{dRW}(X/\W(R))\otimes \Q
$$
that induces ${\rm Tr}^{-1}$ after evaluation at $\{1\}$. In other words, there is a unique system $(e_S)_{S}$ with 
$e_S\in H^{2d}_{dRW}(X/\W_S(R))\otimes \Q$ such that
\begin{enumerate}
\item $\pi_{S,T}(e_S)=e_T$ for all $T\subset S$, where $\pi_{S,T}$ is induced by the projection $$H^{2d}_{dRW}(X/\W_S(R))\xr{} H^{2d}_{dRW}(X/\W_T(R)),$$
\item $\phi_n(e_S)=n^d\cdot e_{S/n}$ for all $n,S$,
\item $e_{\{1\}}={\rm Tr}^{-1}(1)$.
\end{enumerate}
Our goal is to show 
\begin{equation}\label{equation-eS-integral-claim}
 e_S\in H^{2d}_{dRW}(X/\W_S(R)) 
\end{equation}
for every finite truncation set $S$. The strategy of the proof will be to show this for $X=\P^d_R$ first. The next step will be to 
prove that (\ref{equation-eS-integral-claim}) is local in $\Spec(R)$. Locally on $\Spec(R)$ we can find generically finite morphisms
to $\P^d$, which can be used together with the explicit description of de Rham-Witt cohomology after completion (Proposition \ref{proposition-de-Rham-Witt-cohomology-limit-blue}) to prove the claim.
  
\emph{3.Step:} Suppose $X=\P^d_R$. For any finite $S$, we get a morphism of $\W_S(R)$-schemes 
$$
g_S:\W_S(\P^d_R)\xr{} \P^d_{\W_S(R)}
$$
induced by $\frac{x_i}{x_j}\mapsto [\frac{x_i}{x_j}]$ on the standard affine covering. The morphisms $g_S$ are compatible with the Frobenius 
morphisms provided that the action on $\P^d_{\W_S(R)}$ is given by $\phi_n^*(x_i)=x_i^n$.

We obtain
\begin{multline*}
g^*:H^d(\P^d_{\W_S(R)},\omega_{\P^d_{\W_S(R)}/\W_S(R)})\xr{} H^d(\W_S(\P^d_R),\Omega^d_{\W_S(\P^d_R)/\W_S(R)}) \\
\xr{} H^d(\W_S(\P^d_R),\W_S\Omega^d_{\P^d_R/R}) \xr{} H^{2d}_{dRW}(\P^d_R/\W_S(R)).   
\end{multline*}
Note that ${\rm Tr}:H^d(\P^d_{\W_S(R)},\omega_{\P^d_{\W_S(R)}/\W_S(R)})\xr{\cong} \W_S(R)$ and $\delta_S:={\rm Tr}^{-1}(1)$ satisfies
$\phi^*_n(\delta_S)=n^d\cdot \delta_{S/n}$. Therefore $e_S=g^*(\delta_S)$, which proves (\ref{equation-eS-integral-claim}) in the case
of a projective space.

\emph{4.Step:} We claim that in order to prove (\ref{equation-eS-integral-claim}) it is sufficient to prove
\begin{equation}\label{equation-eS-integral-on-maximal-ideal}
  e_S\in H^{2d}_{dRW}(X/\W_S(R))\otimes_{\W_S(R)}\W_S(R_{\mf{m}}) 
\end{equation}
for every maximal ideal $\mf{m}$. Indeed, let $\mathscr{F}$ be the coherent sheaf on $\Spec(\W_S(R))$ associated to $M:=H^{2d}_{dRW}(X/\W_S(R))$.
For every $\mf{m}$, we can choose an open affine neighborhood $U_{\mf{m}}\subset \Spec(R)$ and a section $e_{\mf{m}}\in \mathscr{F}(\W_S(U_{\mf{m}}))$ mapping to 
$e_S\in M\otimes_{\W_S(R)}\W_S(R_{\mf{m}})$. The section $e_{\mf{m}}$ is unique and the sections $(e_{\mf{m}})_{\mf{m}}$ glue to a section of $\mathscr{F}$ on $\bigcup_{\mf{m}} \W_S(U_{\mf{m}})=\W_S(\Spec(R))$, which proves the claim. 

Let $\hat{R}$ be the completion $\varprojlim_{j} R/\mf{m}^j$. For every integer $n$, we have 
$$
n\hat{R}\cap R_{\mf{m}} =\bigcap_{j=1}^{\infty} (nR_{\mf{m}} + \mf{m}^j)=nR_{\mf{m}},
$$ 
and thus $(R_{\mf{m}}\otimes_{\Z} \Q)\cap \hat{R}=R_{\mf{m}}$ as intersection in $\hat{R}\otimes_{\Z}\Q$. Therefore
\begin{equation}\label{equation-eS-integral-on-maximal-ideal-completion}
  e_S\in H^{2d}_{dRW}(X/\W_S(R))\otimes_{\W_S(R)}\W_S(\hat{R}) 
\end{equation}
implies (\ref{equation-eS-integral-on-maximal-ideal}).

\emph{5.Step:} We will show (\ref{equation-eS-integral-on-maximal-ideal-completion}). 
We may pass from $\Spec(R)$ to a neighborhood $\Spec(R')$ of $\mf{m}$. Let $R'$ be such that there exists a generically finite $R'$-morphism 
$$
f:X\times_{\Spec(R)}\Spec(R')\xr{} \P^d_{R'}.
$$
The existence is proved in Proposition \ref{proposition-gen-finite-to-Pd} below. Then $e_S=\frac{1}{\deg(f)}f^*(e_S)$, because the 
classes $(\frac{1}{\deg(f)}f^*(e_S))_S$ satisfy the properties listed in the second step. 

Set $p={\rm char}(R/\mf{m})$. To prove (\ref{equation-eS-integral-on-maximal-ideal-completion}) we may assume that $S=\{1,p,\dots,p^{n-1}\}$. Then Proposition
\ref{proposition-de-Rham-Witt-cohomology-limit-blue}(ii) yields the claim, because for the de Rham cohomology we know that $f^*({\rm Tr}^{-1}(1))$
is divisible by $\deg(f)$. 
\end{proof}
\end{proposition}

\begin{proposition}\label{proposition-gen-finite-to-Pd}
  Let $Y$ be of finite type over $\Spec(\Z)$. Let $X/Y$ be  smooth projective such that $X$ is connected of relative dimension $d$. 
For every closed point $y\in Y$ there is open neighborhood $U$ of $y$, and a generically finite $U$-morphism $X\times_{Y}U\xr{} \P^d_U$.   
\end{proposition}

In order to prove Proposition \ref{proposition-gen-finite-to-Pd} we will need a sequence of lemmas.

\begin{lemma}\label{lemma-gen-hyperplane-section}
  Let $R$ be a local noetherian ring. Let $X/R$ be a smooth projective $R$-scheme such that every connected component 
of $X$ has relative dimension $d\geq 0$ over $\Spec(R)$. Let $\mathscr{L}$ be a relative ample
line bundle. There is $n>0$ satisfying the following property: for every $k\geq 1$ there is a section  $s\in H^0(X,\mathscr{L}^{\otimes kn})$ such that $V(s)$ is smooth of relative dimension $d-1$ over $R$.  
\begin{proof}
  Let $y\in \Spec(R)$ denote the closed point. For $n\gg 0$ we have $H^i(X_y,\mathscr{L}_{\mid X_y}^{\otimes n})=0$ for all $i>0$. By semicontinuity 
we get $H^i(X,\mathscr{L}^{\otimes n})=0$ for all $i>0$, and 
$$
H^0(X,\mathscr{L}^{\otimes n}) \xr{} H^0(X_y,\mathscr{L}_{\mid X_y}^{\otimes n}) 
$$
is surjective. Replace $\mathscr{L}$ by a power such that this holds for all $n\geq 1$.

If the residue field of $R$ is infinite then we can find a section $s_y\in H^0(X_y,\mathscr{L}_{\mid X_y}^{\otimes n})$
such that $V(s_y)$ is smooth of dimension $d-1$. 
In the case of a finite residue field we have to use \cite{P} and may have to replace $\mathscr{L}$ by a high enough power again. 

Let $s$ be a lifting of $s_y$ to $H^0(X,\mathscr{L}^{\otimes n})$, set $H:=V(s)$. If $d=0$ then $H$ is empty, because it has empty intersection with 
the special fibre. For $d\geq 1$, $H$ is flat by the local criterion for flatness, because it has transversal intersection with the special fibre.
 Since $H\xr{} \Spec(R)$ is flat and the special fibre is smooth, we conclude that $H$ is smooth. By Chevalley's theorem, $H$ is of relative dimension $d-1$.
\end{proof}
\end{lemma}

\begin{remark}
  $H$ is empty if and only if $d=0$. 
\end{remark}

\begin{lemma} \label{lemma-gen-finite}
Assumptions as in Lemma \ref{lemma-gen-hyperplane-section}. There is $n\geq 1$ and sections $s_0,\dots,s_d\in H^0(X,\mathscr{L}^{\otimes n})$
such that 
\begin{enumerate}
\item $\bigcap_{i=0}^d V(s_i)$ is empty, 
\item $\bigcap_{i=1}^{d} V(s_i)$ is finite over $R$ and non-empty.
\end{enumerate}
\begin{proof}
  Let $m$ and $s\in H^0(X,\mathscr{L}^{\otimes m})$ such that $H=V(s)$ is a smooth hypersurface as in Lemma \ref{lemma-gen-hyperplane-section}.
Without loss of generality $m=1$.
For $k\gg 0$, we get a surjective map 
$$
H^0(X,\mathscr{L}^{\otimes k})\xr{} H^0(H,\mathscr{L}_{\mid H}^{\otimes k}).
$$ 
By induction on $d$ we can find $s_{H,0},\dots,s_{H,d-1}\in H^0(H,\mathscr{L}_{\mid H}^{\otimes k})$, for some $k\geq 1$, satisfying the desired properties for $H$. Note that $s^j_{H,0},\dots,s^j_{H,d-1}$, for all $j\geq 1$, also satisfy the properties, hence we may suppose $k\gg 0$.
Choose some liftings $s_0,s_1,\dots,s_{d-1}\in H^0(X,\mathscr{L}^{\otimes k})$. Then $s_0,s_1,\dots,s_{d-1},s^k$ satisfy the required properties.
\end{proof}
\end{lemma}

\begin{proof}[Proof of Proposition \ref{proposition-gen-finite-to-Pd}]
Let $\mathscr{L}$ be a relative ample line bundle. Apply Lemma \ref{lemma-gen-finite} to the local ring of $Y$ at $y$. The sections $s_0,\dots,s_d$
extend to $X\times_{Y}\Spec(U)$ for an open affine neighborhood $U$ of $y$. After possibly shrinking $U$ we have $\bigcap_{i=0}^d V(s_i)=\emptyset$ so that
$$
X\times_{Y}\Spec(U) \xr{} \P^d_{U},
$$
defined by $s_0,\dots,s_d$, is well-defined. The second property of Lemma \ref{lemma-gen-finite} implies that the morphism is generically finite.
\end{proof}

\begin{corollary}\label{corollary-Poincare-duality-made-simple}
  Let $R$ be a smooth $\Z$-algebra. Let $X\xr{} \Spec(R)$ be a smooth projective morphism such that $H^*_{dR}(X/R)$ is a projective $R$-module. 
Suppose that $X$ is connected of relative dimension $d$. 
 If the canonical map 
\begin{equation}\label{equation-poincare-duality-deRham}
H^{i}_{dR}(X/R)\xr{} \Hom_R(H^{2d-i}_{dR}(X/R),R)  
\end{equation}
is an isomorphism, then 
\begin{equation} \label{equation-poincare-duality}
H^{i}_{dRW}(X/\W(R))\xr{\cong} \iHom(H^{2d-i}_{dRW}(X/\W(R)),\mathbf{1}(-d)).  
\end{equation}
\begin{proof}
  In view of Proposition \ref{proposition-trace-map} and (\ref{equation-product-varities-C-R}) we get a morphism in $\mathcal{C}_R$:
$$
H^{i}_{dRW}(X/\W(R))\otimes H^{2d-i}_{dRW}(X/\W(R)) \xr{} H^{2d}_{dRW}(X/\W(R)) \xr{} \mathbf{1}(-d)
$$
inducing (\ref{equation-poincare-duality}). Now, $T(\ref{equation-poincare-duality})=(\ref{equation-poincare-duality-deRham})$ proves the claim.
\end{proof}
\end{corollary}

\begin{remark}
Note that the map 
\begin{equation*} 
H^{i}_{dR}(X/R)\xr{} \Hom_R(H^{2d-i}_{dR}(X/R),R)  
\end{equation*}
induced by the pairing 
$$
H^i_{dR}(X/R)\otimes_R H^{2d-i}_{dR}(X/R) \xr{} H^{2d}_{dR}(X/R)\xr{} R
$$
is an isomorphism if for every closed point $y\in \Spec(R)$
the  Hodge-to-de-Rham spectral sequence for the fibre at $y$, 
\begin{equation} \label{equation-Hodge-to-de-Rham-remark}
H^{j}(X_y,\Omega^i_{X_y/k(y)})\Rightarrow H^{i+j}_{dR}(X_y/k(y)),  
\end{equation}
degenerates.
Indeed, since the de Rham cohomology is locally free, it is
also stable under base change and it suffices to show that, for every closed 
point $y\in \Spec(R)$, the Poincar\'e pairing for the fibre at $y$,
$$
H^i_{dR}(X_y/k(y))\otimes_{k(y)} H^{2d-i}_{dR}(X_y/k(y)) \xr{} H^{2d}_{dR}(X_y/k(y))\xr{} k(y),
$$
is non-degenerate. This follows easily from the degeneration of the Hodge-to-de-Rham spectral sequence and Serre duality.  

The degeneration of the spectral sequence (\ref{equation-Hodge-to-de-Rham-remark}) is known in the following cases:
\begin{itemize}
\item $H^j(X,\Omega^i_{X/R})$ is  torsion-free for all $i,j$,
\item $\dim X_{y}\leq {\rm char}(k(y))$.
\end{itemize}

As an example, we have abelian schemes or curves over $R$. 
\end{remark}


\end{document}